\documentclass{article}

\usepackage[letterpaper,top=2cm,bottom=2cm,left=3cm,right=3cm,marginparwidth=1.75cm]{geometry}

\usepackage{graphicx}
\usepackage{amsfonts,amsmath,amssymb,amsthm,amscd}
\usepackage{mathrsfs}
\usepackage{mathtools}
\usepackage{indentfirst}
\usepackage{placeins}
\usepackage{listings}
\usepackage{caption}
\usepackage{subcaption}
\usepackage{xcolor}
\usepackage{comment}
\usepackage{url}
\usepackage{bm}
\usepackage[toc,page]{appendix} 
\usepackage{fancyvrb} 
\usepackage{multirow}
\usepackage{mathtools,nccmath}
\usepackage{cancel} 
\usepackage{enumitem} 
\usepackage[normalem]{ulem} 
\usepackage{fullpage}
\usepackage{dsfont} 
\usepackage{soul}

\usepackage[pdfencoding=auto,psdextra]{hyperref} 
\hypersetup{
  colorlinks   = true, 
  urlcolor     = black, 
  linkcolor    = blue, 
  citecolor   = blue 
}

\newcommand{\R}{\mathbb{R}}
\newcommand{\eps}{\varepsilon}
\newcommand{\cL}{\mathcal{L}}
\newcommand{\cN}{\mathcal{N}}

\theoremstyle{definition}
\newtheorem{defi}{Definition}[section]

\theoremstyle{remark}
\newtheorem{rem}[defi]{Remark}
\usepackage{etoolbox} 
\AtEndEnvironment{rem}{\qed}
\AtEndEnvironment{example}{\qed}

\theoremstyle{plain}
\newtheorem{theorem}[defi]{Theorem}
\newtheorem{lemma}[defi]{Lemma}
\newtheorem{corollary}[defi]{Corollary}
\newtheorem{prop}[defi]{Proposition}

\def\E{\mathbb{E}}

\def\rank{\mathrm{rank}}
\def\tr{\mathrm{Tr}}

\title{Affine constraints in non-reversible diffusions with degenerate noise}
\author{Carsten Hartmann\thanks{Institut f\"ur Mathematik, Brandenburgische Technische Universit\"at Cottbus-Senftenberg, Konrad-Wachsmann-Allee 1, D-03046 Cottbus, Germany.\ Email:\href{mailto:carsten.hartmann@b-tu.de}{carsten.hartmann@b-tu.de}},  
Lara Neureither\thanks{Institut f\"ur Mathematik, Brandenburgische Technische Universit\"at Cottbus-Senftenberg, Konrad-Wachsmann-Allee 1, D-03046 Cottbus, Germany.\ Email:\href{mailto:neurelar@b-tu.de}{neurelar@b-tu.de}}, 
Upanshu Sharma\thanks{School of Mathematics and Statistics, University of New South Wales, Sydney 2052, Australia.\ Email:\href{mailto:upanshu.sharma@unsw.edu.au}{upanshu.sharma@unsw.edu.au}}}

\allowdisplaybreaks

\begin{document}

\maketitle

\tableofcontents

\newpage

\begin{abstract}
    This paper deals with the realisation of affine constraints on nonreversible stochastic differential equations (SDE) by strong confining forces. We prove that the confined dynamics converges pathwise and on bounded time intervals to the solution of a projected SDE in the limit of infinitely strong confinement, where the projection is explicitly given and depends on the choice of the confinement force. We present results for linear Ornstein-Uhlenbeck (OU) processes, but they straightforwardly generalise to nonlinear SDEs.  
    Moreover, for linear OU processes that admit a unique invariant measure, we discuss conditions under which the limit also preserves the long-term properties of the SDE. More precisely, we discuss
    choices for the design of the confinement force which in the limit yield a projected dynamics with invariant measure that agrees with the conditional invariant measure of the unconstrained processes for the given constraint. 
    The theoretical findings are illustrated with suitable numerical examples. 
\end{abstract}

\section{Introduction}

The realisation of constraints by strong confining forces is a classical theme in mechanics; see e.g.~\cite{bornemann1997homogenization,eldering2016realizing,kozlov1990realization} and the references therein. 
Recently there has been a growing interest in studying constrained stochastic differential equations \cite{projection_diffusion,WalterHartmannMaddocks11} and Monte Carlo methods for manifolds \cite{hartmann2008ergodic,xu2024monte}, due to their relevance in molecular dynamics \cite{CiccottiKapralEijnden05,LelievreRoussetStoltz12}, material science \cite{hinch1994brownian}, computational statistics \cite{girolami2011riemann},  or machine learning \cite{pmlr-v22-brubaker12,leimkuhler-constraint-regularization-nn}. 

This article is concerned with the realisation of algebraic  constraints on stochastic differential equations with degenerate noise. Specifically, we consider linear diffusions of Ornstein-Uhlenbeck type with a stiff confinement term that penalises deviations of the stochastic dynamics from the constraint surface. In order to avoid getting lost in details on the differential geometry of submanifolds, we confine our attention to affine subspaces, i.e.~affine constraints. In doing so, we focus on two aspects: (1) the pathwise approximation of an unconstrained diffusion with a strong confining force on a bounded time interval, also called the ``softly constrained system'' in what follows, by a projected or (hard) constrained linear diffusion, (2) the preservation of the underlying Gaussian invariant measure under constraining and strong confinement.   

Available works so far have focused primarily on overdamped Langevin dynamics~\cite{projection_diffusion,SharmaZhang21}, the second reference including additional rotational effects, and underdamped Langevin systems with stiff potentials \cite{reich2000smoothed}. 
Overdamped Langevin systems are rather special, in that they rely on a gradient structure: in its natural coordinates, the overdamped Langevin dynamics is a gradient flow that is perturbed by uncorrelated white noise, whereas the underdamped Langevin dynamics is based on a Hamiltonian structure that is perturbed by noise and dissipation that are balanced in a specific way. The linear systems considered in this paper are not restricted to this class. 
This has consequences for the realisation of constraints on such systems as there is no simple mechanism of imposing constraints by adding strong gradient forces that preserve the invariant measure under the restriction imposed by the constraint.

\paragraph*{Brittleness of non-reversible and degenerate noise dynamics under constraining.}
To illustrate the observation that the structure of the noise may interfere with the realisation of constraints, consider the simple linear system
\begin{align*}
dX^1_t & = X^2_t dt\\
dX^2_t & = -(X^1_t + X^2_t) dt + \sqrt{2}dW_t
\end{align*}
where $W_t$ is a Brownian motion, which has a unique Gaussian invariant measure with zero mean and unit covariance $\mu= \mathcal{N}(0,I)$. Now, let us impose a constraint on any of the two variables. If we fix $X^1_t=b$, the second equation turns into
\[
    dX^2_t = -(b + X^2_t) dt + \sqrt{2}dW_t\,.
\]      
The process $X^2$ does not convergence to the Gaussian distribution of the full system conditioned on $X^1=b$, which is given by $\mu_c= \mathcal{N}(0,I)$, as one could expect, but to a Gaussian measure with unit variance and mean $-b$. On the other hand, if we keep the other variable fixed, $X^2_t=c$, then the resulting dynamics for the first variable, 
\[
X^1_t = X^1_0 + c t\,,
\] 
has no invariant measure at all. This observation should be contrasted with the situation for a reversible system. Consider, for example, the  two-dimensional process $X=(X^1,X^2)$ governed by  
\[
    dX_t = -A X_t dt + \sqrt{2}dW_t
\]
for a symmetric positive definite $2\times 2$-matrix $A$ and Brownian motion $W=(W^1,W^2)$ with two uncorrelated components. 
The unique invariant measure is a zero-mean Gaussian with covariance  $\Sigma=A^{-1}$, and the conditional distributions for $X^2$ given $X^1=b$ and $X^1$ given $X^2=c$ agree with the invariant measures of the corresponding constrained dynamics. The latter follows from the simple fact, that the noise coefficient is the same in both equations and the drift term is the gradient of the potential \[
H(x) = \frac{1}{2}x^T A x = \frac{1}{2}x^T\Sigma^{-1}x\,,
\]
where the invariant measure has a density proportional to $\exp(-H)$. 

Note that the robustness of the invariant measure of reversible dynamics with uncorrelated noise under constraining remains true beyond the simple linear setting that we consider here.

\paragraph{Contribution and novelty}

In this paper we consider linear diffusions subject to affine constraints, but we emphasize that the key results from Section \ref{sec:softOU} can be readily generalised to systems with nonlinear drift (cf.~Remarks \ref{rem:genOU-IC} and \ref{rem:QuantNonlin}). 
While it may appear overly restrictive to consider only the linear case, our work extends available results for nonlinear diffusions on Riemannian submanifolds in several regards. First, our analysis is for linear diffusions with degenerate noise (that may or may not have an invariant measure), while most of the available works, such as \cite{projection_diffusion,SharmaZhang21} apply to the non-degenerate case only. Second, we prove explicit convergence rates for pathwise convergence in the mean square sense, whereas the strongest result so far known to us \cite{Katzenberger91} only discusses qualitative convergence of the stopped process in probability and on bounded time intervals; confining ourselves to linear systems considerably simplifies the  proofs as compared to the rather technical arguments in \cite{Katzenberger91} that utilize relative compactness in a Skorokhod topology. While the results in \cite{Katzenberger91} apply to general c\`adl\`ag processes, our convergence result includes the treatment of transient initial layers in case the initial conditions do not converge to the constraint subspace. Third, we give an explicit characterisation of the limit dynamics in terms of an ambient space formulation involving a projection, 
which is typically only available for overdamped (i.e.\ reversible) Langevin dynamics ~\cite{projection_diffusion}. In this setting, the restriction of the unconstrained dynamics to the constraint subspace is by orthogonal projection. Yet, in contrast to the reversible case with uncorrelated noise, the orthogonal projection may not preserve the invariant measure. Related ambient space formulations also exist for underdamped Langevin dynamics when dealing with hard constraints~\cite{LelievreRoussetStoltz12,WalterHartmannMaddocks11}.
Additionally, we derive explicit expressions for confining forces and resulting oblique projections that preserve either (conditional) moments or the invariant measure in the strong confinement limit; this choice is not unique, moreover the fact that the projection is oblique implies that the resulting constraint forces have a non-vanishing tangential component that affects the dynamics on the constraint surface. As a consequence, the constraint forces violate d'Alemberts principle from classical mechanics, which postulates that forces acting on a constrained system are the same as the tangential components of the corresponding unconstrained system (cf.~\cite{RubinUngar57,bornemann1997homogenization}).

Studying the linear case is not only relevant from a conceptual point of view, but it is also relevant for applications as non-reversible MCMC methods for high-dimensional Gaussians is an active field of research, see e.g.~\cite{lelievre2013optimal,ottobre2020optimal,zhang2022efficient}. 
It has been argued that non-reversibility can accelerate convergence to equilibrium and therefore the speed of convergence of MCMC methods \cite{HwangHwangSheu93,LelievreNierPavliotis13,bierkens2016non,song2022irreversible}. What further motivates our study, is that constraints can be used to cope with ill-conditioned (i.e.~stiff) sampling or stochastic optimisation problems. In applications, soft constraints play the role of a penalisation that introduces numerical stiffness; introducing hard constraints instead of soft ones can reduce stiffness and allow for using larger time steps in the numerical discretisations; see e.g. \cite{roy2017constrained,langmore2023hamiltonian}.   

We should mention that the pathwise convergence analysis in this paper heavily relies on asymptotic stability in that we exploit that the matrix exponential associated with the softly constrained system of stochastic differential equations is dissipative in directions away from the constraint subspace and converges to a projected matrix exponential that is associated with the system under hard constraints. Our analysis is therefore complementary to existing works on weak convergence of conservative  Hamiltonian systems, e.g.~\cite{bornemann1997homogenization,klar2021second} that are non-dissipative and become highly oscillatory in the strong confinement limit, excluding a pathwise analysis. We briefly discuss the relevance of this observation for the underdamped Langevin equation in the course of this paper. Details will be dealt with in a companion paper.

\paragraph{Outline of the article}

The rest of the article is structured as follows. Section \ref{sec:prelim} records the basic notation used throughout this article. Section \ref{sec:softOU} contains the first of the two main results, the pathwise convergence analysis of the softly constrained Ornstein-Uhlenbeck process, together with a review of related results by Katzenberger (Section \ref{ssec:Katz}) and an example where our approach as well as these related results fail (Section \ref{ssec:LinearLangevin}). The second main result, the choice of the oblique projection matrix that guarantees that the constrained dynamics admits the desired long-term stability with respect to the target invariant measure is the content of Section \ref{sec:SteadyState-K}. Specifically, we discuss the connection between the preservation of the (conditional) invariant measure and the choice of the projection matrix (Section \ref{ssec:obliqueProjection}) as well as the connection between the confinement force and the resulting projection (Section \ref{ssec:MatrixK}). The theoretical findings are illustrated with suitable numerical examples in Section \ref{sec:numEx}. Conclusions are drawn in Section \ref{sec:conclusions}. The article contains an appendix that records various technical lemma and auxiliary statements.

\section{Notation}\label{sec:prelim}

Throughout the paper, we split coordinates, vector fields and matrices into constrained and unconstrained variables as follows: A state vector $x\in\R^d$ with $k$ constrained variables is split into $x=(x^1,x^2)^T \in\R^k\times\R^{d-k}$. We typically understand $x\in\R^d$ as a column vector, but we nevertheless write $x=(x^1,x^2)\in\R^k\times\R^{d-k}$ for a row and $x=(x^1,x^2)^T\in\R^k\times\R^{d-k}$ for a column vector, i.e.~$x^1$ and $x^2$ can be either rows or columns, depending on the context. We denote by $X_t=(X_t^1,X_t^2)$ the state of a process $(X_t)_{t\ge 0}$ at time $t\ge 0$. In the same vein, we use the following block-matrix notation and partition a matrix $S\in \R^{d\times d}$ according to 
    \begin{equation}\label{not:BlockMat1}
        S= \begin{pmatrix}
        S_{11} & S_{12} \\  S_{21} & S_{22}
        \end{pmatrix} \in \R^{d \times d}
    \end{equation}
    with 
      \begin{equation}\label{not:BlockMat2}
S_{11} \in \R^{k \times k}, \ S_{12} \in \R^{k \times (d-k)}, \ S_{21} \in \R^{(d-k) \times k}, \text{ and } \ S_{22} \in \R^{(d-k) \times (d- k)}.
    \end{equation}
We write $X^\eps\to Y$ for a softly constrained process  $(X^\eps_t)_{t\ge 0}$ that converges to a process $(Y_t)_{t\ge 0}$ in the limit $\eps\to 0$ where the process $Y$ is subject to a hard constraint. The precise mode of convergence will be specified in the results. For the sake of readability, the confinement parameter $\eps>0$ will be omitted when it plays no role, i.e.~we mostly write $X^\eps_t=(X^1_t,X^2_t)$ instead of $X^\eps_t=(X^{\eps,1}_t,X^{\eps,2}_t)$ where $X^1_t$ and $X^2_t$ both may depend on $\eps$. 

Orthogonal or oblique projection (matrices) are denoted by $P$,  e.g.
\[
Px=\begin{pmatrix}
    0 \\ x^2
\end{pmatrix}\,,
\]
with $P=P^2$ being an idempotent linear operator or $d\times d$-matrix.  The nabla operator 
\[
\nabla =\left(\frac{\partial}{\partial x_1},\,\ldots ,\,\frac{\partial}{\partial x_d}\right)^T\,,
\]
with $x_1,\ldots,x_d$ being the components of $x^1=(x_1,\ldots,x_k)$ and $x^2=(x_{k+1},\ldots,x_d)$, 
is understood as a column vector, i.e. the gradient $\nabla f$ of a 
differentiable function $f\colon\R^d\to\R$ is a column vector, whereas $\nabla\cdot G$ denotes the divergence of a (smooth) vector field $G$, with $x\cdot y=x^Ty$ denoting the inner product between two vectors $x$ and $y$. 
Finally, we write $A:B=\tr(A^TB)$ to denote the Frobenius inner product of two square matrices $A$ and $B$, and  
$\|M\|^2_F = M:M$
for the Frobenius norm of a  matrix $M$.  A matrix $M$ is \emph{Hurwitz} if its spectrum lies in the open left plane, i.e.\ its eigenvalues have strictly negative real parts. Finally, we use the notation $M
>0$ for a square positive definite matrix, i.e.\ if $v^TMv>0$ for any $v \neq 0$.

\section{Convergence results for softly constrained OU processes}\label{sec:softOU}

In this section we study the behaviour of linear constraints when applied to general OU processes. Here linear constraint refers to a mapping 
\begin{equation*}
    \xi\colon \R^d\to \R^k, \ k \leq d, \  \xi(x)=x^1 - b, 
\end{equation*}
where $b\in \R^k$ is a given constant vector and we have used the notation $ x = \left(x^1,\, x^2\right)^T
\in \R^{k}\times\R^{d-k}$. In the case $b\equiv 0$, the map $\xi$ is simply a coordinate projection. We have made the choice to include $b$ in the map $\xi$ since it has important implications regarding sampling of conditional measures discussed in the next section (see Proposition~\ref{prop:genOU-SteSt} and ensuing discussion). 
Our work in this section straightforwardly generalises to affine maps $\xi$ and non-zero mean OU processes as opposed to zero-mean OU process~\eqref{eq:genOU-SC} studied below (see Remark~\ref{rem:genOU-affine-gen}).  The restriction to coordinate-projection constraints is for simplicity of presentation.

We are interested in the $\eps\to 0$ limit of the SDE
\begin{equation} \label{eq:genOU-SC}
    dX_t = MX_t dt - \frac{1}{2\eps} K \nabla|\xi(X_t)|^2 dt + \sqrt{2}C dW_t,
\end{equation}
where $M,K\in\R^{d\times d}$ and $C\in \R^{d\times n}$ are given matrices. Following the notation above we write $X_t\in \R^{d}$ as $X_t=\left(
X^1_t, \,X^2_t\right)^T\in \R^{k}\times\R^{d-k}$. Finally, $W_t$ is a standard Brownian motion in $\R^n$. Note that the matrix $K$ encodes the way in which the zero level-set $\xi^{-1}(0)=\{x\in\R^d:\xi(x)=0\}$ is approached. For instance, if we only focus on the stiff part of the dynamics and choose $K=I$ then we approach the zero level-set via a gradient descent (see~\cite{projection_diffusion,Zhang20} for related ideas). On the other hand, with $K=J-A$ (for $J=-J^T$ and $A=A^T>0$ -- see Section~\ref{sec:SteadyState-K} for details) one expects a spiralling descent~\cite{SharmaZhang21}. See Remark~\ref{rem:genOU-proj} for more details.

The following theorem discusses the soft-constrained limit, $\eps\to 0$, in~\eqref{eq:genOU-SC}. This result uses the block-matrix notation introduced in~\eqref{not:BlockMat1}--\eqref{not:BlockMat2}. 
\begin{theorem}\label{thm:genOU-SC}
Given $\eps>0$, let $X_t^\eps$ be the solution to~\eqref{eq:genOU-SC} with (random) initial datum $X^\eps|_{t=0}=X^\eps_0$. Assume that 
\begin{enumerate}
    \item   $-K_{11} \in\R^{k\times k}$ is Hurwitz,
    \item  $X_0^\eps\to X_0$ in probability as $\eps \to 0$.
\end{enumerate}
Then for any $t>0$ we have 
\begin{equation*}
    X^\eps_t \xrightarrow{\eps \to 0} Y_t \text{ in probability}
\end{equation*}
where $Y_t$ is the solution to the SDE (in $\R^d$) 
\begin{equation}\label{eq:genOU-SC-limit}
    dY_t = PM Y_t dt + \sqrt{2 } PC dW_t, 
\end{equation} 
with $P\in\R^{d\times d}$ defined as 
\begin{equation}\label{def:genOU-Proj}
    P\coloneqq 
    \begin{pmatrix}
        0_{k\times k } & 0_{ k\times (d-k)} \\ \alpha & I_{(d-k)\times (d-k)}    
    \end{pmatrix}, \text{ \ where  \ } \alpha \coloneqq -K_{21}K^{-1}_{11} \in \R^{(d-k)\times k}
\end{equation}
and initial data 
\begin{equation*}
    Y|_{t=0}=P \biggl( X_0 - \begin{pmatrix} b \\ 0\end{pmatrix} \biggr) + \begin{pmatrix} b \\ 0\end{pmatrix}.
\end{equation*}
Using  $Y_t=\left(Y^1_t,\,Y^2_t\right)^T$ where $Y^1_t\in\R^{k}$ and $Y^2_t\in \R^{d-k}$, the limit~\eqref{eq:genOU-SC-limit} can explicitly be written as
\begin{equation}\label{eq:genOU-SC-limit-Expl}
    \begin{aligned}
        dY_t^1 &= 0  \\
        dY_t^2 & =  (\alpha M_{11} + M_{21} ) Y^1_t + (\alpha M_{12} + M_{22} ) Y^2_t dt + \sqrt{2} \hat C dW_t, 
    \end{aligned}  
\end{equation}
where $\hat C = (\alpha C_{11} + C_{21} \ \ \alpha C_{12} + C_{22}) \in \R^{(d-k) \times n}$ and initial data
\begin{equation*}
    Y^1|_{t=0} = b, \ \ Y^2|_{t=0} = \alpha X_0^1 + X_0^2 - \alpha b.
\end{equation*}
Here $X_0=\left(X_0^1,\, X_0^2\right)^T$ is the prescribed limiting initial data for the sequence $X^\eps_0$. In particular,  $Y^1_t\equiv Y^1|_{t=0}$ and therefore  we find
\[\xi(Y_t) = 0, \text{ for any } t\geq 0\,,\]    
i.e.\ the constraint $\xi(Y_t)=0$ is satisfied for every $t$. 
\end{theorem}
Before we prove the above result, let us discuss   the role of the projection map characterised via $P\in\R^{d\times d}$ in \eqref{def:genOU-Proj} in the following remark.
\begin{rem}\label{rem:genOU-proj} 
    The map $P\colon \R^d\to \R^d$ characterised via the matrix $P\in \R^{d\times d}$ defined in~\eqref{def:genOU-Proj} is a projection mapping, i.e.\ $P^2=P$. Furthermore, this projection is orthogonal with respect to the standard inner product, i.e.\ $P^T=P$, if and only if $K_{21}=0$. It turns out that this projections is characterised by the long-time limit of the stiff dynamics in~\eqref{eq:genOU-SC}. More precisely, consider the ODE
    \begin{align*}
        & \frac{d\psi(z,t)}{dt} = -\frac12 K\nabla\bigl|\xi\bigl(\psi(z,t)\bigr)\bigr|^2, \ \ \psi(z,0)=z, \ \ z \in \R^d, \\
    \text{and } \quad    & \theta(z)\coloneqq\lim_{t\to\infty} \psi(z,t),   \ \ z \in \R^d,
    \end{align*}
    so that\ $\psi$ is the flow driven by the constraint and $\theta$
    is the long-time limit of this flow. Then we have the relation $P=\nabla\theta^T$. For a detailed discussion on this connection see Section~\ref{ssec:Katz}. 
\end{rem}

We point out that Katzenberger~\cite{Katzenberger91} studies the qualitative behaviour of general SDEs with stiff drifts (characterised by the presence of $\eps>0$ as above). We provide an alternative proof for several reasons. First, our proof is considerably simpler because we work with linear SDEs and affine constraints. 
Second, we can give an explicit formulation of the limit dynamics, which involves a projection matrix $P$. Third, we prove convergence pointwise in time, for any $t>0$, which enables us to see how the dynamics approach the constraint manifold $\xi^{-1}(0)$ for initial data not satisfying the constraint. Finally, we even have a quantitative result (Theorem \ref{thm:pathwiseconvrate}) for convergence in a pathwise sense as well as pointwise in time. For a detailed comparison to~\cite{Katzenberger91} see Section~\ref{ssec:Katz} and Remark \ref{rem:genOU-IC} below.

Our proof makes use of an asymptotic result for matrix exponentials which we state below for convenience. This result requires the notion of the index of a matrix and the notion of semistable matrices. The index of a square matrix $M$ is the smallest integer $j \in\mathbb N$ such that $\rank(M^{j+1})=\rank(M^j)$. A matrix $M$ is semistable if it has an index $0$ or $1$ and its non-zero eigenvalues have negative real part.

\begin{lemma}[{\cite[Theorem 1]{campbell1979singular}}] \label{thm:CampbellRose}
    Suppose that $M_1,M_2\in\R^{n\times n}$.  Define $N^\eps_t\coloneqq e^{(M_1+\frac1\eps M_2)t}$. Then $N^\eps_t$ converges pointwise as $\eps\to 0$ for $t>0$ if and only if $M_2$ is semistable. Furthermore, if $M_2$ is semistable, then
    \begin{equation*}
        \lim_{\eps\to 0}e^{(M_1+\frac1\eps M_2)t} = e^{(I-M_2M_2^\dag)M_1 t} (I-M_2M_2^\dag),
    \end{equation*}
    where $M^\dag$ is the Drazin inverse of a square matrix $M$, which is the unique matrix that satisfies $M^\dag MM^\dag=M^\dag$, $MM^\dag=M^\dag M$ and $M^{\ell+1}M^\dag=M^{\ell}$ for any $\ell$ larger than or equal to the index of $M$. 
\end{lemma}

\begin{proof}[Proof of Theorem~\ref{thm:genOU-SC}]
We split the proof into two steps. In the first step we consider $b=0$ and in the second step we use the first step to deal with $b\neq 0$.

(1) Consider the case $b=0$. Here  
\begin{equation}\label{eq:genOU-ExplConst}
    \frac12 \nabla |\xi(x)|^2 = \begin{pmatrix} x^1 \\ 0 \end{pmatrix} = \begin{pmatrix}
        I_{k \times k} & 0_{k \times (d-k)} \\ 0_{(d-k)\times k} & 0_{(d-k)\times (d-k)}
    \end{pmatrix} x.
\end{equation}
With the notation 
\begin{equation*}
    \tilde K  
    \coloneqq K \begin{pmatrix}
        I_{k \times k} & 0_{k \times (d-k)} \\ 0_{(d-k)\times k} & 0_{(d-k)\times (d-k)}
    \end{pmatrix} = \begin{pmatrix}
        K_{11}  & 0 \\ K_{21}  & 0
    \end{pmatrix}, \
\end{equation*}
where $K_{11} \in \R^{k \times k}, \ K_{21} \in \R^{(d-k) \times k}$ and the zero matrices are of appropriate dimensions
we can write~\eqref{eq:genOU-SC} as
\begin{equation}\label{eq:genOU-SC-AltForm}
    dX_t = \Bigl( M -\frac{1}{\eps}\tilde K\Bigr)  X_t dt  + \sqrt{2}CdW_t.
\end{equation}
Using variation of constants (see Section~\ref{app:VarConst}), for any $\eps>0$, the explicit solution to~\eqref{eq:genOU-SC} is 
\begin{align}\label{eq:genOU-VarOfCon}
        X^\eps_t = e^{(M-\frac{1}{\eps} \tilde K)t}X_0^\eps +  \sqrt{2}\int_0^t  e^{(M-\frac{1}{\eps} \tilde K)(t-s)} C dW_s. 
\end{align}
 We now show that $-\tilde K$ is semistable which will allow us to make use of Lemma~\ref{thm:CampbellRose} (see 
discussion above Lemma~\ref{thm:CampbellRose} for precise definitions of the index and semistability  of a matrix). Observe that $\rank(\tilde K) = \rank(K_{11}) = k$ by the assumption on $K_{11}$. Moreover 
\[
\tilde K^2 = \begin{pmatrix}K_{11}^2  & \ 0 \\ K_{21} K_{11}  & \ 0\end{pmatrix}
\] 
and therefore $\rank(\tilde K^2)=\rank(K_{11}^2)=\rank(K_{11})=k$ where we use the Sylvester's rank inequality. Thus $\tilde K$ has index $1$. 
Next observe that the kernel of $\tilde K$ is at least  $d-k$ dimensional. Define $v_i \in \R^k \,, i \in \left\{1,\ldots,k\right\}$ to be right eigenvectors of $K_{11}\,,$ i.e. $v_i^T K_{11} = \lambda_i v_i^T$, where $\mathrm{Re}(\lambda_i) > 0$ by assumption. Then $\tilde v^T_i \tilde K = \lambda_i \tilde v_i^T$, where $\tilde v^T_i = (v_i^T,0) \in \R^d$. Hence all non-zero eigenvalues of $-\tilde K$ have strictly negative real parts and it follows that $-\tilde K$ is semistable. 

It is easily checked that $\tilde K^\dag \in\R^{d\times d}$, given by
\begin{equation*}
    \tilde K^\dag = \begin{pmatrix}
    K_{11} ^{-1} & 0 \\ K_{21} K_{11}^{-1} K_{11}^{-1} & 0 
    \end{pmatrix}
\end{equation*}
is the Drazin inverse of $\tilde K$ and $(-\tilde K)^\dag=-(\tilde K^\dag)$. Applying Lemma~\ref{thm:CampbellRose} we have the pointwise convergence
\begin{equation*}
    \lim_{\eps\to 0} e^{(M-\frac1\eps \tilde K)t }  = e^{PM t}P,   \ \text{ where } \ P=I-\tilde K \tilde K^\dag = \begin{pmatrix}
    0 & 0 \\ - K_{21} K_{11}^{-1} & I
    \end{pmatrix},
\end{equation*}
for any $t>0$. Therefore, for the first term in the right hand side of~\eqref{eq:genOU-VarOfCon} we have
\begin{equation*}
    \lim_{\eps\to 0}e^{(M-\frac1\eps \tilde K)t}X_0^\eps = \lim_{\eps\to 0}e^{(M-\frac1\eps \tilde K)t} \ \lim_{\eps\to 0}X_0^\eps = e^{PMt} PX_0 \ \ \text{in probability for } t>0, 
\end{equation*}
where we have used the assumption that $X^\eps_0\to X_0$ in probability.

Next, define the difference of the stochastic integrals, one coming from the limit dynamics and the other from the soft constrained dynamics, as
\begin{align*}
    Z_s^{\eps} = \int_0^s  e^{(M-\frac{1}{\eps} \tilde K)(s-r)} C dW_r - \int_0^s  e^{PM(s-r)} P C dW_r\,.
\end{align*}
The Doob's inequality states that for any $\delta>0$, the martingale $Z^\eps_s$ satisfies
\begin{align}\label{eq:doobineq}
    \mathbb{P} \biggl( \sup_{s \leq t} | Z_s^{\eps}| > \delta \biggr) \leq \frac{\E\bigl[|Z_t^{\eps}|^2\bigr]}{\delta^2}\,.
\end{align}
Using It\^o's isometry we calculate 
\begin{align*}
    \E\bigl[|Z_t^{\eps}|^2\bigr] = \E\biggl[\biggl|\int_0^t \Bigl( e^{(M-\frac{1}{\eps} \tilde K)(t-r)} C  - e^{PM(t-r)} P C\Bigr) dW_r \biggr|^2\biggr] 
    = \E\biggl[\int_0^t  \bigl\| e^{(M-\frac{1}{\eps} \tilde K)(t-r)} C -  e^{PM(t-r)} P C \bigr\|^2_F dr\biggr]\,,
\end{align*}
where $\|\cdot\|_F$ denotes the Frobenius norm of a matrix.

Since for any $t>0$ we have $e^{(M-\frac1\eps \tilde K)t} \to e^{PMt} P$ 
as $\eps \to 0$ (see Lemma \ref{thm:CampbellRose}) we also have that $\| e^{(M-\frac{1}{\eps} \tilde K)(t-r)} C -  e^{PM(t-r)} P C \|^2_F \to 0$ as $\eps\to 0$. Since this is a converging sequence for any $t>0$, it is bounded as well and  therefore by dominated convergence theorem it follows that 
\begin{align*}
     \E\bigl[|Z_t^{\eps}|^2\bigr] \to 0 \text{ as } \eps \to 0.
\end{align*}
By \eqref{eq:doobineq} it follows that
    \begin{equation*}
        \int_0^t  e^{(M-\frac{1}{\eps} \tilde K)(t-s)} C dW_s \xrightarrow{\eps\to 0} 
        \int_0^t  e^{PM(t-s)} PC dW_s  \ \ \text{in probability,}
    \end{equation*}
    and overall $X^{\eps}_t \to Y_t $ in probability, where 
    \begin{equation*}
    d Y_t = PM Y_t dt + \sqrt{2}PC dW_t,
\end{equation*}
with initial condition $Y_0 = PX_0$. Here we have used variation of constants to arrive at the strong form of $Y_t$.

(2) Now we generalise to the case of $b\neq 0$, i.e. $\xi(x) = x^1-b$ and 
\begin{equation}
    \frac12 \nabla |\xi(x)|^2 = \begin{pmatrix} x^1 -b \\ 0 \end{pmatrix} = \begin{pmatrix}
        I_{k \times k} & 0_{k \times (d-k)} \\ 0_{(d-k)\times k} & 0_{(d-k)\times (d-k)}
    \end{pmatrix} \biggl[ x - \begin{pmatrix}
        b\\ 0
    \end{pmatrix}\biggr].
\end{equation}
The soft constrained dynamics then reads
\begin{equation}\label{eq:genOU-SC-b}
    dX_t = M   X_t dt  -\frac{1}{\eps}\tilde K \biggl(X_t - \begin{pmatrix}
        b \\ 0
    \end{pmatrix} \biggr) + \sqrt{2}CdW_t.
\end{equation}
Now, consider the coordinate shift 
\[ 
    \bar X = X - \begin{pmatrix}
     b \\ 0
    \end{pmatrix}
\]
which transforms~\eqref{eq:genOU-SC-b} to 
\begin{equation*} 
    d \bar X_t^\eps 
    = M\biggl[ \bar X_t^\eps + \begin{pmatrix}
    b \\ 0 
    \end{pmatrix} \biggr] dt - \frac{1}{\eps} \tilde K 
    \bar X_t^\eps  dt + \sqrt{2  } C dW_t.
\end{equation*}
In the new variables the solution reads
\begin{equation*} 
     \bar X_t^\eps 
    = e^{(M-\frac{1}{\eps} \tilde K)t}\bar X^\eps_0 + \int_0^t  e^{(M-\frac{1}{\eps} \tilde K)(t-s)} M\begin{pmatrix}
        b \\ 0 
    \end{pmatrix} ds +   \sqrt{2}\int_0^t  e^{(M-\frac{1}{\eps} \tilde K)(t-s)} C dW_s .
\end{equation*}
Now we are back to the previous case, i.e.\ by the same argument as above we arrive at
\begin{equation*} 
     \bar Y_t := \lim\limits_{\eps \to 0}\bar X_t^\eps 
    = e^{PM}P\bar X_0 + \int_0^t  e^{PM(t-s)} PM\begin{pmatrix}
         b \\ 0 
    \end{pmatrix} ds +   \sqrt{2}\int_0^t  e^{PM(t-s)} PC dW_s  
\end{equation*}
in probability. As in the previous case
\begin{equation*}
    d\bar Y_t = PM \biggl[\bar Y_t + \begin{pmatrix}
        b \\ 0 
    \end{pmatrix} \biggr] dt + \sqrt{2}PC dW_t
\end{equation*}
with initial condition 
\begin{equation*}
    \bar Y_0 = P\bar X_0.
\end{equation*}
Transforming back using $Y \coloneqq \bar Y + \left(b ,\, 0\right)^T$
we find
\begin{equation*}
    d Y_t = PM Y_t dt + \sqrt{2}PC dW_t
\end{equation*}
with initial condition 
\begin{equation*}
    Y_0 = \bar Y_0 + \begin{pmatrix}
     b \\ 0
\end{pmatrix} = P \biggl[X_0 - \begin{pmatrix}
     b \\ 0
\end{pmatrix} \biggr] + \begin{pmatrix}
     b \\ 0
\end{pmatrix}.
\end{equation*}
The explicit form~\eqref{eq:genOU-SC-limit-Expl} follows by inserting the matrix form of $P$ into the limit.
\end{proof}
The following two remarks discuss two aspects related to Theorem~\ref{thm:genOU-SC}: convergence at initial times and generalisation  to nonlinear SDEs.
\begin{rem}[Initial conditions and convergence in $C([0,\infty))$ ]\label{rem:genOU-IC}
    Theorem~\ref{thm:genOU-SC} provides the convergence $X^\eps_t \to Y_t $ as $\eps \to 0$ for any $t>0$.  This convergence statement is in fact wrong for $t=0$ since $\lim_{\eps\to 0}X^{\eps}_0 = Y_0 $ if and only if $X_0  \in \xi^{-1}(0), $ where $Y_0 \coloneqq P \big( X_0 - \left( b,\, 0 
    \right)^T \big) + \left( b,\, 0 
    \right)^T $.
    Katzenberger~\cite[Theorem 6.3]{Katzenberger91} addresses this issue by studying the convergence of a modified process $Y^\eps_t \coloneqq X^\eps_t - \psi(X^\eps_0,\frac{t}{\eps}) + \theta(X^\eps_0)$ where $\psi$ defined in~\eqref{eq:katz-ODE} is the ODE flow corresponding to the stiff part of the SDE and $\theta$ is its long-time limit~\eqref{eq:theta}. By construction $\psi(X^\eps_0,0)=X^\eps_0$ and therefore at $t=0$, $Y^{\eps}_0 = \theta(X^\eps_0) \in \xi^{-1}(0)$. Using this modified process, uniform convergence on bounded time intervals in $C([0,\infty))$  follows, which leads to pathwise convergence in probability.
    We prefer to work with the original soft constrained process \eqref{eq:genOU-SC} and study its limit. Of course, we could work with an analogue construction here and get the same kind of convergence as in \cite{Katzenberger91}. Theorem~\ref{thm:pathwiseconvrate} below implies the convergence result of Katzenberger for the unmodified process, if the limiting initial datum lies on the constraint manifold, i.e.\ $X_0 \in \xi^{-1}(0)$. 
 \end{rem}

\begin{rem}[Generalisation to nonlinear drifts]
    The proof for Theorem~\ref{thm:genOU-SC} can be easily generalised to the class of nonlinear SDEs where the drift can be written as a sum of a linear function and a smooth, bounded, nonlinear perturbation, see~\eqref{eq:SC-Nonlin-gen} below. This implies that the drift is Lipschitz which is typically required for well-posedness of strong solutions for SDEs. The major difficulty here is that we require compactness of the sequence $(X^\eps)_{\eps>0}\in C([0,\infty))$, which can be extracted via standard approaches (see~\cite[Section 4]{Katzenberger91} for references).  
    More precisely, let $X^\eps_t$ solve
    \begin{align}\label{eq:SC-Nonlin-gen}
        dX^\eps_t = MX^\eps_t + f(X^\eps_t) - \frac{1}{\eps} \tilde K X^\eps_t dt + \sqrt{2} C dW_t
    \end{align}
    where $f\colon \R^d \to \R^d$ is smooth and uniformly bounded  and the other coefficients are as before. Here the limiting dynamics $Y_t$ solves
     \begin{align*}
        dY_t = PMY_t + Pf(Y_t)dt + \sqrt{2} PC dW_t,
    \end{align*}
    where $P=\left(\begin{smallmatrix}
        0 & 0 \\ \alpha & I
    \end{smallmatrix} \right)$ as before. 
    By variation of constants (see Proposition \ref{prop:varofconst}), $X_t^\eps$ admits the solution 
    \begin{align*}
               X^\eps_t = e^{(M-\frac{1}{\eps} \tilde K)t}X_0^\eps + \int_0^t  e^{(M-\frac{1}{\eps} \tilde K)(t-s)} f(X^\eps_s) ds +  \sqrt{2}\int_0^t  e^{(M-\frac{1}{\eps} \tilde K)(t-s)} C dW_s, 
    \end{align*}
    and the convergence of the first and last term above to the corresponding terms in $Y_t$ follows as in the proof of Theorem~\ref{thm:genOU-SC}. Assuming that $\left(X^\eps_s\right)$ is compact in $C([0,t])$, i.e.~converges up to  subsequences, by dominated convergence we expect
    \begin{align*}
      \lim\limits_{\eps \to 0}  \int_0^t  e^{(M-\frac{1}{\eps} \tilde K)(t-s)} f(X^\eps_s) ds  = \int_0^t \lim\limits_{\eps \to 0}   e^{(M-\frac{1}{\eps} \tilde K)(t-s)} f(X^\eps_s) ds  = \int_0^t P f(Y_s) ds,
    \end{align*}
    and therefore $X^\eps_t \to Y_t$ in probability as in the proof above. Consequently, Theorem~\ref{thm:genOU-SC} generalises to a considerably larger class of SDEs. However, we do not provide a complete proof as this requires us to discuss technical results regarding compactness of the sequence $\left(X^\eps\right)$ which are outside the scope of this work; see~\cite[Section 5]{Katzenberger91} and~\cite{KurtzProtter91} for details. 
\end{rem}

\subsection{Quantitative bounds and convergence rates}

We introduce a quantitative convergence result (see Theorem~\ref{thm:pathwiseconvrate} below)  which generalises Theorem~\ref{thm:genOU-SC} by employing the following crucial lemma along with a Gronwall inequality argument. 

\begin{lemma}[Matrix exponential]\label{lem:matrixexp}
Consider the matrix $\tilde K=\begin{pmatrix} K_{11} & 0 \\ K_{21} & 0 \end{pmatrix}\in \R^{d\times d}$ where $K_{11} \in \R^{k \times k}$ is invertible and $K_{21} \in \R^{(d-k) \times k}$. 
\begin{enumerate}[label=(\roman*)]
        \item \label{lem:matrixexpstatement} Then
        \begin{align} \label{eq:matrixexp}
	e^{-\tilde Kt}        
        	= \begin{pmatrix}
            e^{- K_{11} t}  & 0 \\ -K_{21} K_{11}^{-1}+ K_{21} e^{- K_{11} t} K_{11}^{-1} & I        
            \end{pmatrix}. 
    	\end{align}
    \item \label{lem:matrixexpconvrate} If $-K_{11}$ is Hurwitz then
    \begin{align*}
    e^{-\frac{1}{\eps}\tilde Kt}\xrightarrow{\eps\to 0} P := \begin{pmatrix}
        0 & 0 \\ - K_{21}K_{11}^{-1} & I
    \end{pmatrix}.
    \end{align*}
    \item  \label{lem:integratedexpmatrix-projection} 
    Assume that $-K_{11}$ is Hurwitz. 
    Let $V \in \mathbb{C}^{k \times k}$ be the matrix that transforms $K_{11}$ to its Jordan normal form, i.e.\ $V^{-1} K_{11} V = \Lambda + N$ where $\Lambda$ is a diagonal matrix with the eigenvalues of $K_{11}$ as entries  and $N$ is an upper triangular nilpotent matrix of order $m \in \mathbb{N}$, i.e.\ $N^m=0$.  
    
    Then for any $t>0$  
    \begin{align} \label{eq:Frobeniusest_expK-P}
        \bigl\| 
	e^{-\frac{1}{\eps}\tilde K t} 
    - P\bigr\|_F^2 \leq c_A e^{-\frac{2\lambda_1t}{\eps} } \sum\limits_{j=0}^{m-1} t^j\frac{\|N^j\|_F^2}{\eps^j }
    \end{align}
    where  $\lambda_1$ is the smallest real part of all eigenvalues of $K_{11}$ and 
    \begin{equation*}
    c_A= \Bigl( 1 + \bigl\| K_{21}\bigr\|_F^2 \bigl\| K^{-1}_{11}\bigr\|_F^2 \Bigr) \kappa(V)  mk. 
    \end{equation*} 
    Here $\kappa(V) \coloneqq \|V\|_F^2 \|V^{-1}\|_F^2$ is the condition number of the matrix $V$.
    
Moreover, 
\begin{align}\label{eq:matrixexp-int-estimate}
    \int_0^t \| e^{-\frac{1}{\eps}\tilde K(t-s)} - P \|_F^2 ds  &\leq \eps c_P,
\end{align}
where 
\begin{equation*}
c_P= \Bigl( 1 + \bigl\| K_{21}\bigr\|_F^2 \bigl\| K^{-1}_{11}\bigr\|_F^2 \Bigr) \kappa(V) k m  \, \sum\limits_{j=0}^{m-1}\frac{\|N^j\|_F^2}{(2\lambda_1)^{j+1}}.
\end{equation*} 

These estimates simplify in two particular cases: 
\begin{enumerate}
    \item\label{item:nonDef} If $K_{11}$ is non-defective then~\eqref{eq:Frobeniusest_expK-P} becomes 
    \begin{align} \label{eq:Frobeniusest_expK-Psymm}
        \bigl\| 
	e^{-\frac{1}{\eps}\tilde K t} 
    - P\bigr\|_F^2 \leq c_A e^{-\frac{2\lambda_1t}{\eps} },
    \end{align}
    with 
    \begin{equation*}
        c_A=\Bigl( 1 + \bigl\| K_{21}\bigr\|_F^2 \bigl\| K^{-1}_{11}\bigr\|_F^2 \Bigr) \kappa(V) k ,
    \end{equation*}
    and estimate~\eqref{eq:matrixexp-int-estimate} stays unchanged with $c_P$ is given by
    \begin{align*}
    c_P = \left(1 + \bigl\| K_{21}\bigr\|_F^2 \bigl\| K^{-1}_{11}\bigr\|_F^2\right) \kappa(V) k \,.
    \end{align*}
    \item\label{item:sym-nonDef} If $K_{11}$ is symmetric and non-defective, then the estimate above apply with $\kappa(V) \equiv 1$ in the constants.
\end{enumerate}
\end{enumerate}
\end{lemma}

\begin{corollary}\label{cor:expK-P}
Let  $\tilde K$ be defined as in Lemma \ref{lem:matrixexp}, $-K_{11} \in \R^{k \times k}$ be Hurwitz and defective. Then for any $\delta \in (0,\lambda_1)$
\begin{align*}
    \bigl\|e^{-\frac{1}{\eps}\tilde K t} 
    - P\bigr\|_F^2 \leq c e^{-2\frac{\lambda_1 - \delta}{\eps} t } 
    \end{align*}
where $c=c(\delta)  = (m-1)^m \left(\frac{\|N\|^2_F}{2\delta e}\right)^{m-1}$ ,$\lambda_1$ is the smallest real part of all eigenvalues of $K_{11}$ 
and $\kappa(V)$, $m$ are as in Lemma \ref{lem:matrixexp} above.
\end{corollary}
The proof of part~\ref{lem:matrixexpstatement} above follows on the  lines of \cite[Proposition 2.9]{SharmaZhang21}. All proofs are presented in Appendix~\ref{app:MatrixExp}. 

We now present the quantitative result of the pathwise convergence. 
\begin{theorem} \label{thm:pathwiseconvrate}
    Given $\eps>0$, let $X^\eps_t$ solve \eqref{eq:genOU-SC} with initial datum $X^\eps_0$ and $Y_t$ solve \eqref{eq:genOU-SC-limit} with initial datum $Y_0=P \big(X_0 -\left( b,\, 0 
    \right)^T \big) +\left( b,\, 0 
    \right)^T $.  

\begin{enumerate}[label=(\roman*)]
    \item\label{item:PathQuant} We have 
    \begin{align}\label{eq:PathBound}
        \E\biggl[\sup_{t\in [0,T]} \bigl|X_t^\eps - Y_t\bigr|^2 \biggr] \leq 
        c_1\biggl( \E\bigl[|X^{1,\eps}_0-b\bigr|^2] + \E\bigl[|X^\eps_0-X_0|^2\bigr] + \eps\bigl[1+   e^{\lambda_{max}(PM)T} \bigr]\biggr) e^{c_2 T^2}
    \end{align}
where $c_1,c_2$ are independent of $\eps$ and $T$, and $\lambda_{\max}(PM)\geq 0$ is the principal eigenvalue of $PM$. 

In particular, if the limiting initial datum is well-prepared, i.e.\ $X_0\in \xi^{-1}(0)$ 
    \begin{align*}
    \E\biggl[\sup_{t\in [0,T]} \bigl|X_t^\eps - Y_t\bigr|^2 \biggr] \leq 
     c_1\biggl(\E\bigl[|X_0^\eps - X_0|^2\bigr]  + \eps\bigl[1+   e^{\lambda_{max}(PM)T} \bigr]\biggr)
    e^{c_2 T^2}
\end{align*} 
and hence  if $\E\bigl[|X_0^\eps - X_0|^2\bigr] \to 0$ as $\eps\to 0$ then
\begin{equation*}
    X^\eps\xrightarrow{\eps\to 0} Y \text{ in probability on } C([0,T];\R^d). 
\end{equation*}

\item\label{item:PointQuant}  Let $K_{11}$ be non-defective. For any $t>0, \ X_0 \in \R^d$ we have 
\begin{align}\label{eq:PointBound}
    \E\Bigl[  \bigl|X^\eps_t - Y_t\bigr|^2 \Bigr] \leq
    c_1\biggl( e^{-\frac{2\lambda_1 t}{\eps}} \E\bigl[ |X_0^{1} - b|^2\bigr] + 2\E\bigl[ | X_0^\eps - X_0|^2 \bigr]  + \eps\bigl[1+   e^{\lambda_{max}(PM)t} \bigr] \biggr) 
    e^{c_2 t^2}.
\end{align}
 Let $K_{11}$ be defective. For any $t>0, \ X_0 \in \R^d$ and  any $\delta \in (0,\lambda_1)$ we have
\begin{align}\label{eq:PointBound-def}
    \E\Bigl[  \bigl|X^\eps_t - Y_t\bigr|^2 \Bigr] \leq
    \tilde c_1\biggl( e^{-\frac{2(\lambda_1-\delta )t}{\eps}} \E\bigl[ |X_0^{1} - b|^2\bigr] + 2\E\bigl[ | X_0^\eps - X_0|^2 \bigr]  + \eps\bigl[1+   e^{\lambda_{max}(PM)t} \bigr] \biggr) 
    e^{c_2 t^2},
\end{align}
where $c_1, \tilde c_1, c_2>0$ are independent of $\eps$ and $t$, and $\lambda_1>0$ is the smallest real part of eigenvalues of $K_{11}$. 
In particular, with initial data $\E\bigl[|X_0^\eps - X_0|^2\bigr] \to 0$, where $X_0 \in \R^d$ is arbitrary, for any $t>0$ we have
\begin{equation*}
    X_t^\eps\xrightarrow{\eps\to 0} Y_t \text{ in probability}. 
\end{equation*}
\end{enumerate}
\end{theorem}

Before presenting the proof, a remark is in order to discuss: (a) role of initial conditions, (b) role of $\lambda_{\max}(PM)$, and (c) related estimates for Fokker-Planck equations.

\begin{rem}[Pathwise versus pointwise]\label{rem:QuantDisc}
    The key difference between the pathwise bound~\eqref{eq:PathBound} and the pointwise-in-time bound~\eqref{eq:PointBound} is the treatment of the initial condition $|X^{1,\eps}_0-b|^2$ and the fact that while the pathwise bound compares the entire trajectory (including $t=0$) the pointwise bound only applies to $t>0$. Due to this key difference the pointwise bound does not see the initial boundary layer (characterized via the exponential decay in the first term) while the pathwise bound does. As a consequence, if $X^\eps_0 \to X_0 \notin \xi^{-1}(0)$, i.e.\ $X^1_0\neq b$, then the pathwise estimate~\eqref{eq:PathBound} does not vanish, but the pointwise estimate~\eqref{eq:PointBound} does. Furthermore, if $X^{1,\eps}_0\to b$ at a rate $f(\eps)$ and $X^{\eps}_0\to X_0$ at rate $g(\eps)$, then $f(\eps)$, $g(\eps)$ determine the rate of convergence in~\eqref{eq:PathBound} as $\eps\to 0$, i.e.\ we have 
    \begin{equation*}
        \E\biggl[\sup_{t\in [0,T]}\bigl|X_t^\eps -Y_t\bigr|^2 \biggr] \leq C\min\bigl\{f(\eps),g(\eps),\eps\bigr\},
    \end{equation*}
    where $C$ is independent of $\eps$. 
    
    Next we discuss the role of $\lambda_{\max}(PM)$ that appears in the estimates above. In the proof of the quantitative estimate we need to control $\E[|Y_t|^2]$ (see discussion on $I_{2}$ in the proof below), where $Y_t$ is the limiting dynamics~\eqref{eq:genOU-SC-limit}.  The kernel of the matrix product $PM \in \R^{d \times d}$ is  at least a $k$-dimensional, by definition of the projection in~\eqref{def:genOU-Proj} and hence $PM$ has $k$ zero eigenvalues. The other eigenvalues of $PM$ agree with the eigenvalues of $(PM)_{22}$ since for any eigenvector $v \in \R^{d-k}$ of $(PM)_{22}$, i.e.\ $ (PM)_{22} v = \lambda v$, we have $ PM (0,v)^T = \lambda (0,v)^T$. Consequently, there are two possible cases. First, in the case where $(PM)_{22}$ is Hurwitz, which corresponds to $Y^2$ in the limiting dynamics~\eqref{eq:genOU-SC-limit} admitting an invariant measure (see Section~\ref{sec:SteadyState-K} for a detailed study of this setting), it follows that $\lambda_{\max}(PM)=0$ and we have $e^{\lambda_{max}(PM)t}=1$ in the bounds above. Second, if $(PM)_{22}$ is not Hurwitz, we have an additional term that grows exponentially in time with rate given by $e^{\lambda_{max}(PM)t}$. It should be noted that similar exponential growth estimates also arise when using alternative methods for controlling $\E[|Y_s|^2]$, for instance see the approach via Gronwall's inequality employed in Remark~\ref{rem:QuantNonlin}.  

    Note that the defective or non-defective nature of $K_{11}$ only plays a role in the pointiwse estimates. In particular, this determines the exponential rate of decay when the initial conditions do not satisfy the constraint. As indicated by the estimates, the case of non-defective $K_{11}$ leads to better convergence rate of the initial data. 

    Let us mention that these quantitative estimates generalise to the case of nonlinear SDEs with Lipschitz drift, see Remark~\ref{rem:QuantNonlin} for a discussion. 
    Finally, we point out that the quantitative pathwise estimates in Theorem~\ref{thm:pathwiseconvrate} also provide estimates on the corresponding Fokker-Planck equations. In particular, using $\nu^\eps_t = \mathrm{law}(X_t^\eps)$ and $\rho_t = \mathrm{law}(Y_t)$, a quantitative estimate of the Wasserstein-2 distance $\mathcal W_2(\nu_t,\rho_t)$ follows since  
    \begin{equation*}
    \bigl(\mathcal W_2 (\nu^\eps_t , \rho_t)\bigr)^2 \leq \E\bigl[|X_t^\eps-Y_t|^2\bigr] \leq \E\biggl[\sup_{t\in [0,T]}|X_t^\eps-Y_t|^2\biggr].
    \end{equation*}
\end{rem}

\begin{proof}[Proof of Theorem \ref{thm:pathwiseconvrate}]
    \ref{item:PathQuant} First consider the case $\xi(x) = x^1$, i.e.\ $b=0$. The solutions to~\eqref{eq:genOU-SC} and~\eqref{eq:genOU-SC-limit} read    
    \begin{align*}
        X^\eps_t &= e^{ - \frac{1}{\eps}\tilde K t} X^\eps_0 + \int_0^t e^{-\frac{1}{\eps} \tilde K(t-s)} MX^\eps_s ds + \int_0^t e^{-\frac{1}{\eps}\tilde K(t-s)} C dW_s \ \text{ and } \\
        Y_t &= Y_0 + \int_0^t PMY_s ds + \int_0^t PC dW_s\,,
    \end{align*}
    where we have employed variation of constants \eqref{eq:VarofCon-sol} to arrive at the solution of $X^\eps_t$ and $Y_t$. Using Young's inequality we find
    \begin{equation}\label{eq:PathEstSep}
    \begin{aligned}
        \E\biggl[ \sup\limits_{t \in [0,T]} \bigl|X^\eps_t - Y_t\bigr|^2 \biggr] 
        &\leq  3\E\biggl[ \sup\limits_{t \in [0,T]} \bigl|e^{-\frac{1}{\eps}\tilde K t}  X^\eps_0 - Y_0\bigr|^2 \biggr] + 3 \E\biggl[ \sup\limits_{t \in [0,T]} \biggl|\int_0^t \Bigl( e^{-\frac{1}{\eps}\tilde K(t-s)}MX^\eps_s - PMY_s\Bigr) ds \biggr|^2 \biggr]  \\
        & \ +3\E\biggl[ \sup\limits_{t \in [0,T]} \biggl|\int_0^t \Bigl( e^{-\frac{1}{\eps} \tilde K(t-s)}C - PC \Bigr)dW_s\biggr|^2 \biggr] \eqqcolon 3 (I_1 + I_2 + I_3) \,.
    \end{aligned}   
    \end{equation}
    We treat each term separately and start with the last one. Applying Doob's inequality followed by the It\^o isometry, then using the sub-multiplicativity of the Frobenius norm and lastly the bound \eqref{eq:matrixexp-int-estimate} we find 
    \begin{align*}
        I_3 &\leq   
        4 \E \biggl[\biggl| \int_0^T \Bigl( e^{- \frac{1}{\eps} \tilde K (T-s) }  - P \Bigr) C dW_s \biggr|^2 \biggr]
        = 4 \int_0^T \Bigl\|\bigl(e^{- \frac{1}{\eps} \tilde K (T-s) }  - P\bigr)C \Bigr\|^2_F ds \\
        &\leq  4 \int_0^T \bigl\|e^{- \frac{1}{\eps} \tilde K (T-s) }  - P \bigr\|_F^2 \|C \|^2_F ds 
        \leq 4 \eps c_P \|C\|_F^2. 
    \end{align*}

    Next we consider $I_1$. Recall that for $b=0$  we have $Y_0 =  PX_0$. Adding a zero and applying Young's inequality in the first step we have
    \begin{equation}\label{eq:Quant-I1-0}
        \bigl|e^{- \frac{1}{\eps} \tilde K t}X_0^\eps - Y_0\bigr|^2 
        \leq 2\bigl|(e^{- \frac{1}{\eps} \tilde K t} - P)X_0^\eps \bigr|^2  + 2\bigl| PX_0^\eps - Y_0\bigr|^2 
        \leq  2\bigl\|e^{- \frac{1}{\eps} \tilde K t} - P\bigr\|_F^2 \bigl|X^{1,\eps}_0\bigr|^2 + 2|PX_0^\eps - PX_0|^2 
    \end{equation}
where the second inequality follows by the sub-multiplicativity of the norm and using~\eqref{eq:matrixexp} which gives
\begin{align*}
        (e^{- \frac{1}{\eps} \tilde K t} - P)X_0^\eps 
        = \begin{pmatrix}
            e^{-\frac{1}{\eps}K_{11}t} & 0 \\ K_{21}e^{-\frac{1}{\eps}K_{11}t}K_{11}^{-1} & 0
        \end{pmatrix}\begin{pmatrix} X^{1,\eps}_0 \\ X^{2,\eps}_0 \end{pmatrix}
        = (e^{- \frac{1}{\eps} \tilde K t} - P) 
        \begin{pmatrix} X^{1,\eps}_0 \\ 0 \end{pmatrix}.
    \end{align*}
As a consequence, using $|PX_0^\eps - PX_0|^2  \leq \|P\|_F^2 |X_0^\eps - X_0|^2 $ and Corollary \ref{cor:expK-P} or \eqref{eq:Frobeniusest_expK-Psymm} depending on whether $K_{11}$ is defective or not, we can bound $I_1$ from above by 
\begin{align} \label{eq:Quant-I1} 
    I_1 \leq c \left( \E(|X_0^{1,\eps}|^2 + \E(|X_0^\eps - X_0|^2) \right) 
\end{align}
where $c>0$ is independent of $T $ and $\eps$.
If additionally $X_0 \in \xi^{-1}(0),$ i.e. $X_0^1=0$, so that \[|X_0^{1,\eps}|^2 = |X_0^{1,\eps} - X_0^1|^2 \leq |X_0^{\eps} - X_0|^2\]
and we have the overall bound
\begin{equation} \label{eq:Quant-I1-X0inxi}
    I_1 \leq c \E\bigl[|X_0^\eps - X_0|^2\bigr], 
\end{equation}
where again $c>0$  is independent of $T$ and $\eps$.
%

For $I_2$ we first add a zero and calculate using Young's inequality and in the second step the Cauchy-Schwarz inequality, 
\begin{align*}
  I_2 &=  \E\biggl[ \sup\limits_{t \in [0,T]} \biggl|\int_0^t \Bigl( e^{-\frac{1}{\eps}\tilde K(t-s)}M (X^\eps_s - Y_s) + (e^{-\frac{1}{\eps}\tilde K (t-s)} - P)MY_s \Bigr)ds \biggr|^2 \biggr] \\
&\leq 2\E\biggl[  \sup\limits_{t \in [0,T]} \biggl|\int_0^t e^{-\frac{1}{\eps}\tilde K(t-s)}M (X^\eps_s - Y_s) ds\biggr|^2\biggr] + 2\E\biggl[ \sup\limits_{t \in [0,T]} \biggl|\int_0^t (e^{-\frac{1}{\eps}\tilde K (t-s)} - P)MY_s ds \biggr|^2 \biggr]\\
&\leq 2\E\biggl[  \sup\limits_{t \in [0,T]} t \int_0^t \biggl|e^{-\frac{1}{\eps}\tilde K(t-s)}M (X^\eps_s - Y_s) \biggr|^2 ds\biggr] \\
&\qquad\qquad\qquad+ 2\E\biggl[ \sup\limits_{t \in [0,T]} \left( \int_0^t \bigl\| (e^{-\frac{1}{\eps}\tilde K (t-s)} - P)M \bigr\|_F^2 ds \right) \left(\int_0^t |Y_s|^2 ds \right)\biggr]\\
&\leq 2\|M\|_F^2 \E\biggl[  \sup\limits_{t \in [0,T]} t \int_0^t \bigl\|e^{-\frac{1}{\eps}\tilde K(t-s)}\bigr\|^2_F \bigl|X^\eps_s - Y_s\bigr|^2  ds \biggr] \\
&\qquad\qquad\qquad+ 2\|M\|_F^2  \E\biggl[  \sup\limits_{t \in [0,T]} \biggl(\int_0^t \bigl\|e^{-\frac{1}{\eps}\tilde K(t-s)}-P\bigr\|_F^2 ds \biggr) \int_0^t |Y_s|^2ds \biggr] \\
&\eqqcolon 2\|M\|_F^2 \bigl(I_{2,1} + I_{2,2}\bigr)\,.
\end{align*}
Let us consider the terms separately. For $I_{2,1}$ we first add a zero and apply Young's inequality, using Fubini's theorem in the third step we find
\begin{equation}\label{eq:I21-Bound}
    \begin{aligned}
    I_{2,1}& \leq  \E\biggl[ \sup\limits_{t \in [0,T]} 2t \int_0^t \Bigl[ \bigl\| e^{-\frac{1}{\eps}\tilde K(t-s)}  - P\bigr\|^2_F + \|P\|^2_F \Bigr]  |X^\eps_s - Y_s|^2 ds \biggr] \\
    &\leq  2T\E\biggl[  \int_0^T \bigl[ \bigl\| e^{-\frac{1}{\eps}\tilde K(T-s)}  - P\bigr\|^2_F  + \|P\|^2_F \bigr] |X^\eps_s - Y_s|^2 ds\biggr] \\
    &= 2T \int_0^T \bigl[ \bigl\| e^{-\frac{1}{\eps}\tilde K(T-s)}  - P\bigr\|^2_F  + \|P\|^2_F \bigr] \E\biggl[ |X^\eps_s - Y_s|^2 \biggr] ds \\
    &\leq   2 T\int_0^T \bigl[ \bigl\| e^{-\frac{1}{\eps}\tilde K(T-s)}  - P\bigr\|^2_F + \|P\|^2_F \bigr] \E\biggl[ \sup_{\tau\in [0,s]}|X^\eps_\tau - Y_\tau|^2 \biggr] ds. 
    \end{aligned}
    \end{equation}
For  $I_{2,2}$  we also use \eqref{eq:matrixexp-int-estimate} in the first step and Fubini's theorem to compute
\begin{align*}
I_{2,2} & \leq \eps c_P \E\biggl[  \sup\limits_{t \in [0,T]} \int_0^t |Y_s|^2ds \biggr] 
=  \eps c_P \int_0^T \E \bigl[ |Y_s|^2 \bigr] ds\,.
\end{align*}
Since $Y_t$ is an OU process we know that $m_s,\Sigma_s$  are given in \eqref{linSDE-NormalSolution}, and therefore 
\begin{align}
    \E\bigl[|Y_s|^2\bigr] &= |m_s|^2 + \tr(\Sigma_s) \notag\\ 
    &= \bigl|e^{PMs}\E\bigl[Y_0\bigr]\bigr|^2 + \tr\bigl(e^{PMs} \Sigma_0 e^{M^TP^Ts}\bigr) + \int_0^s \tr\bigl(e^{PM(s-r)} PCC^TP^T e^{M^TP^T(s-r)}\bigr) dr \notag\\
    &\leq \bigl\|e^{PMs}\bigr\|_F^2 |\E\bigl[Y_0\bigr]|^2 + \|e^{PMs}\|_F^2 \| \Sigma_0^{\frac{1}{2}} \|^2_F + \int_0^s \|e^{PM(s-r)}\|_F^2 \| PC\|_F^2  dr\,. \label{eq:2ndmoment-Y}
\end{align}
Let $S$ be the matrix that transforms $PM$ into its Jordan normal form, i.e. $SPMS^{-1} = \Lambda + N,$ where $\Lambda$ is a diagonal matrix with the eigenvalues of $PM$ as entries and $N$ is an upper triangular nilpotent matrix. Then \begin{align*}
    \|e^{PMs}\|_F^2 &= \| S^{-1}e^{\Lambda s} e^{N s}S\|_F^2 \leq  \| S^{-1} \|_F^2 \|e^{\Lambda s}\|_F^2 \|e^{N s}\|_F^2 \|S\|_F^2 \leq  d \kappa(S) \|e^{N s}\|_F^2 e^{2\lambda_{\max(PM)}s}\,,
\end{align*}
where $\kappa(S)$ is the condition number of $S$.
Hence, by a similar calculation as \eqref{eq:calcintexptimespolynomial}, we find
\begin{align*}
    \int_0^T \|e^{PMs}\|_F^2 ds \leq \hat c e^{\lambda_{\max}(PM)T}. 
\end{align*}
for some constant $\hat c>0$ independent of $\eps$ and $T$.

This means, that overall $I_{2,2} \leq \eps c  e^{\lambda_{max}(PM) T}$, where $c>0$ is independent of $T$ and $\eps$ but depends on $\E[Y_0]$, $\Sigma_0$ (and consequently $\E[X_0]$ by the definition of $Y_0$) through the first term in~\eqref{eq:2ndmoment-Y}.  
Substituting these bounds into~\eqref{eq:PathEstSep}, we find
\begin{align*}
    \E\biggl[\sup_{t\in [0,T]} \bigl|X_t^\eps - Y_t\bigr|^2 \biggr]
    &\leq \gamma(\eps,T) + c_2 T \int_0^T\bigl[ \bigl\| e^{-\frac{1}{\eps}\tilde K(T-s)}  - P\bigr\|^2_F + \|P\|^2_F \bigr] \E\biggl[\sup_{t\in [0,s]} \bigl|X_t^\eps - Y_t\bigr|^2 \biggr] ds  ,
\end{align*}
where $\gamma(\eps,t) = c_1 (\E(|X_0^{1,\eps}|^2 + \E(|X_0^{\eps} - X_0|^2 + \eps (1+e^{\lambda_{max}(PM) T})) $ for $X_0 \in \R^d$ and $\gamma(\eps,t) =  c_1 ( \E(|X_0^{\eps} - X_0|^2) + \eps (1+e^{\lambda_{max}(PM) T})) $ for $X_0 \in \xi^{-1}(0)$.
Thus, by Gronwall's inequality and \eqref{eq:matrixexp-int-estimate} 
\begin{align*}
    \E\biggl[\sup_{t\in [0,T]} \bigl|X_t^\eps - Y_t\bigr|^2 \biggr] \leq 
   \gamma(\eps,T)
    e^{c_3 T^2 + \eps c_P c_2 T}.
\end{align*}

Next we discuss the case $b\neq 0$, i.e.\ $\xi(x) = x^1 - b$. 
With the coordinate-shifted variables $\bar X_t = X_t - \left(b,\,0
\right)^T$  and $\bar Y_t = Y_t +  \left(b,\, 0\right)^T$, we are back to the previous case of a coordinate projection onto zero and note that 
\[
\E\biggl[\sup\limits_{t \in [0,T]} \bigl|\bar X_t - \bar Y_t\bigr|^2\biggr]= \E\biggl[\sup\limits_{t \in [0,T]} \bigl|X_t -  Y_t\bigr|^2\biggr]\,.
\]
The proof goes through exactly as above, in particular the estimates for $I_1, I_2, I_3$ remain unchanged. Let us briefly discuss $I_1$, since this is the term where the shifted initial datum appears. First note that \eqref{eq:Quant-I1-0} becomes (after  adding a zero and using that $\bar Y_0 = P \bar X_0 $)
    \begin{align*}
        \bigl|e^{- \frac{1}{\eps} \tilde K t}\bar X_0^\eps - \bar Y_0\bigr|^2  
        &= \bigl|e^{- \frac{1}{\eps} \tilde K t}\bar X_0^\eps - P \bar X_0^\eps   + P \bar X_0^\eps - \bar Y_0\bigr|^2 
         \leq 2\bigl| \left(e^{- \frac{1}{\eps} \tilde K t} - P \right) \left(X_0^\eps -\left(\begin{smallmatrix}
            b \\ 0
        \end{smallmatrix}\right) \right) \bigr|^2 +  \bigl|  P\bar X_0^\eps- P \bar X_0 \bigr|^2\\
        & \leq \|e^{-\frac{}{\eps}\tilde K t} - P\|_F^2 \bigl|X^{1,\eps}_0 - b \bigr|^2 + \|P\|_F^2 \bigl|  X_0^\eps- X_0 \bigr|^2
    \end{align*}
    Now, since $X_0^1=b$, we find essentially the same estimate for $I_1$, which is given by \eqref{eq:Quant-I1-X0inxi}.
The only change is an additional term, called $I_4$ below. It arises due to the shift in the initial conditions and is deterministic. We bound it as follows
\begin{align*}
   I_4:= &\E\biggl[\sup\limits_{t \in [0,T]} \biggl|\int_0^t \bigl(e^{-\frac{1}{\eps} \tilde K (t-s)} - P\bigr) M \begin{pmatrix}
        b \\ 0
    \end{pmatrix} ds \biggr|^2\biggr] 
    \leq  \sup\limits_{t \in [0,T]} t \int_0^t \bigl\|e^{-\frac{1}{\eps} \tilde K (t-s)} - P\bigr\|^2_F \|M\|^2_F |b|^2 ds  \\
    &\leq \sup\limits_{t \in [0,T]} t \|M\|^2_F |b|^2 c_P \eps =  T \|M\|^2_F |b|^2 c_P \eps 
\end{align*}
where we have used~\eqref{eq:matrixexp-int-estimate}
and which adds up to the constant. Note that also the constant factors change, since Young's inequality is applied to four summands in this case, so that \eqref{eq:PathEstSep} will have a factor of 4 now.   
The convergence in probability follows from the bound above and Markov's inequality. This completes the proof of~\ref{item:PathQuant} .

Next we prove the pointwise in time estimate \ref{item:PointQuant}. We have 
\begin{align*}
        \E\Bigl[  \bigl|X^\eps_t - Y_t\bigr|^2 \Bigr] 
        &\leq  4\E\biggl[  \bigl|e^{-\frac{1}{\eps}\tilde K t}  X^\eps_0 - Y_0\bigr|^2 \biggr] + 4 \E\biggl[\biggl|\int_0^t \Bigl( e^{-\frac{1}{\eps}\tilde K(t-s)}MX^\eps_s - PMY_s\Bigr) ds \biggr|^2 \biggr]  \\
        & \ +4\E\biggl[ \biggl|\int_0^t \Bigl( e^{-\frac{1}{\eps} \tilde K(t-s)}C - PC \Bigr)dW_s\biggr|^2 \biggr] + 4\biggl|\int_0^t \bigl(e^{-\frac{1}{\eps} \tilde K (t-s)} - P\bigr) M \begin{pmatrix}
        b \\ 0
    \end{pmatrix} ds \biggr|^2 \\
    &\eqqcolon 4 (\bar I_1 +\bar I_2 + \bar I_3 + \bar I_4)\,,
    \end{align*}
and by repeating the calculation in (\ref{eq:Quant-I1-0}-\ref{eq:Quant-I1}), we obtain  
\begin{align*}
\bar I_1 \leq 
2\bigl\|e^{- \frac{1}{\eps} \tilde K t} - P\bigr\|_F^2 \E\bigl[ |X_0^{1} - b|^2\bigr] +  2 \|P\|_F^2\E\bigl[ | X_0^\eps - X_0|^2 \bigr]\,.
\end{align*}   
Depending on whether $K_{11}$ is defective or not, we apply the bound of Corollary \ref{cor:expK-P} or \eqref{eq:Frobeniusest_expK-Psymm} to the term $\bigl\|e^{- \frac{1}{\eps} \tilde K t} - P\bigr\|_F^2 $.\\
From It\^o's isometry along with~\eqref{eq:matrixexp-int-estimate} we bound $\bar I_3$ by 
\begin{align*}
\bar I_3 \leq \eps c_P \|C\|_F^2. 
\end{align*}
Repeating the calculations for $I_4$ and $I_2$, we find
\begin{align*}
   \bar I_4
   \leq  \eps t \|M\|^2_F |b|^2 c_P 
\end{align*}
and
\begin{align*}
 \bar I_2 &\leq 2\|M\|_F^2 \left( 2t \int_0^t (\|e^{-\frac{1}{\eps }\tilde K(t-s)} - P \|_F^2 + \|P\|_F^2 )  \E\bigl[\bigl|X^\eps_s - Y_s\bigr|^2\bigr] ds  + \eps c  e^{2\lambda_{\max}(PM)t} \right)  . 
\end{align*}

Finally, combining these bounds and applying Gronwall's inequality we arrive at (in the case that $K_{11}$ non-defective)
\begin{align*}
    \E\Bigl[  \bigl|X^\eps_t - Y_t\bigr|^2 \Bigr] \leq
    c_1\biggl( e^{-\frac{2\lambda_1 t}{\eps}} \E\bigl[ |X_0^{1} - b|^2\bigr] + \E\bigl[ | X_0^\eps - X_0|^2 \bigr]  + \eps\bigl[1+   e^{\lambda_{max}(PM)t} \bigr]  \biggr) 
    e^{\eps t + c_2 t^2},
\end{align*}
where $c_1,c_2$ are independent of $\eps,t$. The bound for $K_{11}$ being defective follows analogously replacing the exponential decay of the initial term by $e^{-\frac{\lambda_1 -\delta}{\eps}t}$ as given by Corollary \ref{cor:expK-P}.
\end{proof}

\begin{rem}[Generalisation to nonlinear drift]\label{rem:QuantNonlin}
The quantitative estimate in Theorem~\ref{thm:pathwiseconvrate}  also works for nonlinear SDEs where the drift $f\colon \R^d \to \R^d$ is smooth and Lipschitz continuous (see below for precise growth conditions). The SDE~\eqref{eq:genOU-SC} then reads 
\begin{equation*}
    dX_t = f(X_t)dt -\frac{1}{2\eps} K\nabla|\xi(X_t)|^2dt +\sqrt{2}CdW_t
\end{equation*}
where the corresponding limit (cf.~\eqref{eq:genOU-SC-limit}) is given by 
\begin{equation*}
       dY_t = P f(Y_t) dt + \sqrt{2 } PC dW_t,
\end{equation*}
with $P$ as defined in~\eqref{def:genOU-Proj}. The only change in Theorem~\ref{thm:pathwiseconvrate} in this case is in the $I_2$ term. In particular, adding and subtracting $e^{-\frac{1}{\eps}\tilde K (t-s)}f(Y_s)$ we can estimate  
\begin{equation}\label{eq:remI2}
\begin{aligned}
    I_2 &= \E\biggl[ \sup\limits_{t \in [0,T]} \biggl|\int_0^t \Bigl( e^{-\frac{1}{\eps}\tilde K(t-s)}\bigl(f(X^\eps_s) -f(Y_s)\bigr)  + (e^{-\frac{1}{\eps}\tilde K (t-s)} - P)f(Y_s) \Bigr)ds \biggr|^2 \biggr]\\
    & \leq 2\E\biggl[ \sup_{t\in [0,T]} \biggl|\int_0^t e^{-\frac{1}{\eps}\tilde K(t-s)}\bigl(f(X^\eps_s) -f(Y_s)\bigr) ds\biggr|^2\biggr] + 2\E\biggl[\sup_{t\in [0,T]}  \biggl|\int_0^t (e^{-\frac{1}{\eps}\tilde K (t-s)} - P)f(Y_s) ds \biggr|^2 \biggr]\\
    &\leq 2 L_f\E\biggl[ \sup_{t\in [0,T]} \biggl|\int_0^t e^{-\frac{1}{\eps}\tilde K(t-s)}\bigl|X^\eps_s -Y_s\bigr| ds\biggr|^2\biggr] + 2  \biggl(\int_0^T \bigl\|e^{-\frac{1}{\eps}\tilde K(T-s)}-P\bigr\|_F^2 ds \biggr) \int_0^T \E\bigl[|f(Y_s)|^2]ds \\
    &\leq 2 T L_f\int_0^T \bigl[ \|e^{-\frac{1}{\eps }\tilde K(T-s)} - P \|_F^2 + \|P\|^2_F \bigr] \E\biggl[ \sup_{\tau\in [0,s]}|X^\eps_\tau - Y_\tau|^2 \biggr] ds + 2\eps  c_P \int_0^T \E\bigl[|f(Y_s)|^2\bigr]ds  
\end{aligned}
\end{equation}
where the first inequality follows by Young's inequality, the second inequality follows since $f$ is Lipschitz with constant $L_f$ and by using Cauchy-Schwarz inequality in the second integral, the third inequality follows from the bound~\eqref{eq:I21-Bound} above and by using~\eqref{eq:matrixexp-int-estimate} for the second integral.

Next we provide a bound for $\E[|f(Y_s)|^2]$. A Lipschitz function $f$ has linear growth at infinity and therefore (under sufficient regularity) we can assume that there exists a constant $c_f>0$ such that $|y\cdot f(y)|\leq c_f(1+|y|^2)$
(or equivalently $|f(y)|\leq c(1+|y|)$ for some constant $c>0$).

In the following we use $\rho_t=\mathrm{law}(Y_t)$ which solves 
\begin{equation*}
   \partial_t\rho_t = -\nabla\cdot \bigl( Pf \rho_t \bigr)+ \nabla^2:\bigl(\bar C \bar C^T \rho_t \bigr)   
\end{equation*}
where $\bar C=PC$, $\nabla^2$ is the Hessian and $A:B=\tr(A^TB)$. We have 
\begin{align*}
    \frac12 \frac{d}{dt}\E\bigl[|Y_s|^2\bigr] &= \frac12\frac{d}{dt} \int_{\R^d} |y|^2 \rho_t(dy) = \int_{\R^d} \Bigl( \frac12 \nabla |y|^2\cdot Pf(y)+ \bar C\bar C^T : \frac12 \nabla^2 ( |y|^2) \Bigr) \rho_t(dy) \\
    & = \int_{\R^d} \Bigl( y\cdot Pf(y)+  \tr(\bar C\bar C^T ) \Bigr) \rho_t(dy) \leq \|P\|_F \int_{\R^d} c_f(1+|y|^2) \rho_t(dy) + \|\bar C\|_F^2 
\end{align*}
where the second equality follows by using the dynamics of $\rho_t$ and applying integration by parts. The inequality now follows by applying the growth bounds on $f$ and using the  Young's inequality. 

Using Gronwall's inequality we find 
\begin{equation*}
    \E\bigl[|f(Y_s)|^2\bigr]\leq c_f(1+\E[|Y_s|^2]) \leq   \E\bigl[ |Y_0|^2\bigr] e^{\|P\|_F c_f t } + 
    \frac{\|\bar C\|^2_{F}}{c_f\|P\|_F}\bigl(\|\bar C\|^2_F e^{\|P\|_F c_f s }-1\bigr) \leq m_1 e^{\|P\|_F c_f s} + m_2 
\end{equation*}
where $m_1,m_2>0$ are independent of $s$.
Substituting this bound back into the bound for $I_2$~\eqref{eq:remI2} and assuming well-prepared initial data
we arrive at the following quantitative estimate for the nonlinear SDE:
\begin{equation*}
     \E\biggl[\sup_{t\in [0,T]} \bigl|X_t^\eps - Y_t\bigr|^2 \biggr] \leq \eps c_1e^{c_2 T^2}. 
\end{equation*}
\end{rem}

\begin{rem}\label{rem:genOU-affine-gen}
    We note that the results in this section directly apply to the general setting of affine constraint and non-zero mean OU processes by recasting the process to the form \eqref{eq:genOU-SC} considered in this paper. For affine constraints, this can be done by applying an appropriate similarity transformation, which transforms the affine constraint to a coordinate projection. More precisely, let $\xi(x) = Bx - b, \ B \in \R^{k \times d}, b \in \R^k$, where $B$ has rank $k$. Then $ \nabla |\xi|^2 =2 B^T(Bx -b)$ and the similarity transformation is given by \[S = \begin{pmatrix}
        V^T & 0 \\ 0 & I
    \end{pmatrix} \in \R^{d \times d},\] where $V$ is the orthonormal matrix that diagonalizes $B$, i.e.\ $V^TB^TBV = \mathrm{diag}(\lambda_1,\ldots,\lambda_k)$.\\  For non-zero OU process one can consider a coordinate shift, which eventually leads to a shifted level set of the constraint map $\xi.$
\end{rem}

\subsection{Comparison to Katzenberger's approach}\label{ssec:Katz}

In this section we compare our results to those of Katzenberger~\cite{Katzenberger91}, who considers the asymptotic problems of the type studied here for general semi-martingales, and in particular the setting of diffusion processes (see~\cite[Section 8]{Katzenberger91}).

We now present the result in~\cite{Katzenberger91} in the language of this article. To this end consider
\begin{equation}\label{eq:genOU-KatzInitProb}
    dX_t = f(X_t) dt + \frac{1}{\eps} F(X_t) dt + \sqrt{2}C dW_t,
\end{equation}
where $f\colon \R^d \to \R^d$ is locally Lipschitz, $C \in \R^{d \times d}$ and $W_t$ is a $d$-dimensional Brownian motion.
The stiff drift term is characterised by the vector field $F\colon \R^d\to\R^d$. In the setting of this paper $f(x)=Mx$ and $F(x) = \frac12 K \nabla |\xi(x)|^2$. 

The main result in~\cite{Katzenberger91} makes use of the following ordinary differential equation:
\begin{equation} \label{eq:katz-ODE}
    \dot \psi(z,t) = F(\psi(z,t)) \in \R^d \,,  \ \psi(z,0) = z  \in \R^d\,.
\end{equation}
Define 
\[\Gamma = \left\{ x \in \R^d : F(x)= 0 \right\} \subset \R^{k}\,\] 
as the  set of fixed points of the ODE or, put differently, the set of points satisfying the constraint (i.e.~$\Gamma=\xi^{-1}(0)$ in the language of the previous paragraphs). Furthermore, for a given initial condition $z\in \R^d$, we define the long-time limit of the ODE above 
\begin{align*}
    \theta(z) :=\lim\limits_{t \to \infty} \psi(z,t) \quad \text{ and } \quad U_{\Gamma} := \left\{x \in \R^d : \theta(x) \text{ exists and } \theta(x) \in \Gamma  \right\}. 
\end{align*}
Before stating the main result in~\cite{Katzenberger91}, we introduce the limit process 
        \begin{equation*}
            dY_t = \nabla\theta^T f(Y_t) + \sqrt{2} \nabla \theta^T \, C dW_t \,,
        \end{equation*}
and the stopping time $\lambda(K) = \left\{ \inf t \geq 0 : Y_t \notin \mathring{K}\right\}$ for a compact set $K \subset \Gamma$, and we write $(Y_t)_{\lambda(K)}$ for the process  stopped at $\lambda(K)$. Here $\mathring{K}$ denotes the interior of $K$.
    \begin{theorem}\cite[Section 8]{Katzenberger91} \label{thm:katzenberger}
        Assume that 
        \begin{enumerate}[label=(\roman*)]
            \item $\forall \ y \in U_{\Gamma}: \ \nabla F(y) \in \R^{d \times d}$ has $k$ eigenvalues with negative real part \label{ass:katz-hurwitz}
            \item $\theta \in C^2$ and $\nabla \theta, \nabla^2 \theta $ are locally Lipschitz 
            \item $X^\eps(0) \to Y(0) \in \Gamma$ in probability.
        \end{enumerate}  
        Then the solution $X^\eps$  to~\eqref{eq:katz-ODE} satisfies 
        \begin{equation*}
            (X^\eps_t)_{\lambda(K)}\to (Y_t)_{\lambda(K)}  \text{ as $\eps\to 0$ } 
        \end{equation*}
        in probability in $C([0,\infty))$ uniformly on bounded time intervals.
    \end{theorem}
    Let us now relate the assumptions of the above Theorem to our assumptions and give explicit expressions for the functions $\psi, \theta$ defining the limiting dynamics $Y$ in our setting. To this end note that  for  $\xi(x) = x^1 - b \in \R^k$ we have
\begin{align*}
    F(x) = -\frac{1}{2} K \nabla |x^1 -b|^2 = - \begin{pmatrix}
        K_{11} & 0 \\ K_{21} & 0
    \end{pmatrix} \left( x
    - \begin{pmatrix}
        b \\ 0 
    \end{pmatrix} \right) \,, \quad \text{and } \ \nabla F \equiv -\begin{pmatrix}
 K_{11} & 0 \\ K_{21} & 0   
\end{pmatrix} \eqqcolon \tilde K.
\end{align*}
Assumption $(i)$ in Theorem~\ref{thm:katzenberger} is equivalent to requiring that $\tilde K$ or equivalently (for an explanation see the proof of Theorem \ref{thm:genOU-SC}) $-K_{11}$ is Hurwitz (as in Theorem~\ref{thm:genOU-SC}), and Assumption $(ii)$ always holds in our setting as will be illustrated below. 

Note that Katzenberger provides an implicit form for the soft-constrained limit defined via $\theta$. In contrast, Theorem~\ref{thm:genOU-SC}
directly states the explicit form of the limiting dynamics which includes a projection matrix $P$ (recall Remark~\ref{rem:genOU-proj} for details). Such projections have recently been studied in related works~\cite{projection_diffusion,Zhang20,SharmaZhang21} which deal with non-degenerate diffusions (note that in our SDEs the diffusion matrix can be degenerate). Note, however, that the results in Theorem~\ref{thm:genOU-SC} entail a pointiwse-in-time limit while~\cite{Katzenberger91} provides convergence of the path over $[0,T]$. This difference is due to the treatment of initial conditions (see Remark~\ref{rem:genOU-IC} for details). The results in Theorem~\ref{thm:pathwiseconvrate} thus are a quantitative version of~\cite{Katzenberger91}. Note that these quantitative results easily generalise to nonlinear SDEs (see Remark~\ref{rem:QuantNonlin}) without requiring any compactness arguments. Additionally, our quantitative pointwise-in-time estimate~\eqref{eq:PointBound} does not require any special treatment of the initial conditions (see Remark~\ref{rem:QuantDisc} for details).


We now examine the limiting dynamics above in the setting of coordinate-projection constraints. In order to explicitly calculate the limit SDE above we need to solve the ODE~\eqref{eq:katz-ODE}, which admits the solution
\begin{align*}
    \psi(z,t) & = e^{-\tilde K t}\psi(z,0)+\int_0^t e^{-\tilde K (t-s)} \tilde K \begin{pmatrix}
        b \\ 0
    \end{pmatrix} ds
     = e^{-\tilde K t}z + \bigl(I-e^{-\tilde K t}\bigr)\begin{pmatrix}
        b \\ 0
    \end{pmatrix},
\end{align*}
and compute it's long-time limit $\theta(z)$. This is contained in Lemma \ref{lem:matrixexp} which states that $e^{-\tilde K t}\to P$ as $t\to\infty$ (this follows by rescaling time by $\eps$), using which we find
\begin{equation} \label{eq:theta}
    \theta(z) = P \left(z  - \begin{pmatrix}
        b \\ 0
    \end{pmatrix}\right)+\begin{pmatrix}
        b \\ 0
    \end{pmatrix} \in \xi^{-1}(0) \ \forall \, z \in \R^d \quad \text{ and } \quad \nabla \theta^T=P.
\end{equation}

This shows that $\theta$ is indeed a projection onto the manifold $\Gamma = \xi^{-1}(0)$ and $P$ is a projection onto the tangent space  of $\xi^{-1}$.

\subsection{Where things fail: underdamped Langevin dynamics with spatial constraint}\label{ssec:LinearLangevin}

So far we have treated soft constraint limits of OU processes with possibly degenerate noise, admitting a unique invariant measure. The underdamped Langevin dynamics with quadratic potential can be seen as a special case of such OU processes. More precisely in the same notation as above, the underdamped Langevin equation can be written as\begin{align*}
    dX_t = (J-A) \nabla H(X_t) dt + \sqrt{2A} dW_t, 
\end{align*}
where  
\begin{align*}
   X&= \begin{pmatrix}
       q \\ p
   \end{pmatrix}\in \R^{2d},\ A= \begin{pmatrix}
        0 & 0 \\ 0 & \gamma I
    \end{pmatrix}\in \R^{2d \times 2d}, \gamma \in \R,\ J= \begin{pmatrix}
        0 & I \\ -I & 0
    \end{pmatrix} \in \R^{2d \times 2d},
\end{align*}
and the quadratic Hamiltonian $H(q,p) = V(q) + \frac{1}{2} |p|^2$, where $V$ is a quadratic potential in $q$.

One physically relevant constraint is given by the zero level set of the spatial CG map $\xi(q)=q^1$. 
A natural choice for the matrix $K$ is $K=(A-J)$; cf. Proposition \ref{prop:genOU-choiceK} below.
The invariant measure of the corresponding soft constrained SDE 
\begin{align*}
    dX_t = (J-A) \nabla H(Z_t) dt - (A-J)\frac{1}{2\eps}\nabla |q^1|^2 +\sqrt{2A} dW_t, 
\end{align*}
is given by $\mu^{\eps} = \frac{1}{Z} e^{-V(q) - \frac{1}{2} |p|^2 - \frac{1}{2\eps}|q^1|^2}$ which 
converges to the conditional Gaussian measure $\mu^{\eps = 0} = \frac{1}{\tilde Z} e^{-V(0,q^2) - \frac{1}{2} |p|^2}$ as $\eps \to 0$. Therefore we expect our limit result in Theorem~\ref{thm:genOU-SC} with $K=A-J$ to hold here. However, we cannot apply Theorem \ref{thm:genOU-SC}
here because the condition that $-K_{11} $ is Hurwitz is not satisfied as $K_{11} = (A-J)_{11} = 0_{k \times k}.$ 
Similarly,  Theorem~\ref{thm:katzenberger} of~\cite{Katzenberger91} does not apply since  
\begin{align*}
    \nabla F = \begin{pmatrix}
        0_{d \times d} & 0_{d \times d} \\ I_{k \times k} & 0_{k \times (2d-k)} \\ 0_{(d-k) \times k} & 0_{(d-k) \times (2d-k)}
    \end{pmatrix}
\end{align*}
and thus all eigenvalues of $\nabla F$ have zero real parts. 
Hence, we have to resort to other methods in order to prove similar limit results for the underdamped Langevin equation. This is the topic of the companion paper~\cite{HartmannNeureitherSharma25}.

\section[Stability of invariant measures]{Stability of invariant measures under hard and soft constraints}\label{sec:SteadyState-K}

The previous section dealt with soft-constrained limits of OU processes without any discussion of the long-time behaviour. However, as stated in the introduction, a goal of soft-constraining  is to sample conditional measures on manifolds. In this section we discuss invariant measures in the context of soft-constraining and the role of the matrix $K$ which characterizes the constraint. We answer several questions: (a) does the limit of the soft-constrained yield the \emph{correct} conditional measure, i.e. is the invariant measure of the projected dynamics \eqref{eq:genOU-SC-limit} the same as the invariant measure of the unconstrained process conditional on the constraint $\xi$, and (b) what choice for $K$ leads to the \emph{correct} conditional measure.

In this section we will focus on OU processes of the type~\eqref{eq:genOU-SC} which admit a unique Gaussian invariant measure. Therefore, we make the following assumptions throughout this section. 
\begin{itemize}
    \item $M$ is Hurwitz; 
    \item $\bigl(M,C\bigr)$ is controllable, i.e. $\mathrm{rank}[C, MC, M^2C,\ldots, M^{d-1}C]=d$. 
\end{itemize}
In case $C$ has full rank, the controllability is given.
Under these assumptions, the unconstrained SDE~\eqref{eq:genOU-SC}, i.e.\ with $K\equiv 0$, admits the unique  invariant measure (see Proposition~\ref{prop:invmeasOU})
\begin{equation*}
    \mu\in \mathcal P(\R^d), \ \ \mu = \mathcal N(0,\Sigma),
\end{equation*}
where $\Sigma \in \R^{d\times d}$ is the unique symmetric positive definite solution to the Lyapunov equation
\begin{equation}\label{eq:genOU-Lyap}
    M\Sigma+\Sigma M^T=-2CC^T.
\end{equation}
It turns out that any  OU process which admits $\mu$ as invariant measure can be rewritten in the form (see Proposition~\ref{prop:genOU-PHS})
\begin{equation} \label{eq:genOU-Ham}
    dX_t = (J-A)\Sigma^{-1}X_t dt + \sqrt{2}C dW_t,
\end{equation}
 where $A\coloneqq CC^T \geq 0 \in\R^{d\times d}$ is symmetric positive semi-definite and $J \coloneqq \frac12 \bigl(-\Sigma M^T+M\Sigma\bigr) = - J^T \in\R^{d\times d}$ is skew symmetric. 
The corresponding soft-constrained version reads 
\begin{equation} \label{eq:genOU-SC-Ham}
    dX_t = (J-A)\Sigma^{-1}X_t dt - \frac{1}{2\eps} K \nabla|\xi(X_t)|^2 dt + \sqrt{2}C dW_t,
\end{equation}
which corresponds to~\eqref{eq:genOU-SC} with the choice $M=(J-A)\Sigma^{-1}$.

We use the particular form~\eqref{eq:genOU-SC-Ham} for two reasons. First, this form allows us to make an educated guess for the crucial matrix $K$ which encodes how the constraint submanifold $\xi^{-1}(0)$ is approached. Second, the softly-constrained (nonlinear) Langevin dynamics studied in the companion paper~\cite{HartmannNeureitherSharma25} can also be written in this form with specific choices of $A$ and $J$ and $\Sigma^{-1}X$ will be replaced by the gradient of a given Hamiltonian.  
The effect of the $K$ matrix will be looked at in detail in Section \ref{ssec:MatrixK} and illustrated with a numerical example of a Langevin-type dynamics in Section \ref{sec:numEx}.

\subsection{Conditional probabilities and constrained dynamics}\label{ssec:obliqueProjection}

Given a CG map $\xi(x)=x^1-b$ we denote the conditional probability measure of $\mu=\mathcal N(0,\Sigma)$ 
restricted to the level set $\xi^{-1}(0)$ by $\mu_c \in \mathcal P(\R^{d-k})$. It is explicitly given by (see for instance~\cite[Section 3.4]{Eaton83}) 
\begin{equation}\label{eq:genOU-condSteSta}
    \mu_c =\mathcal N(m_c,\Sigma_c) \ \text{ where } \ 
    m_c = \Sigma_{21}\Sigma_{11}^{-1} b \  \text{ and } \
    \Sigma_c = \Sigma_{22} - \Sigma_{21}\Sigma_{11}^{-1} \Sigma_{12}.
\end{equation}
Let us recall the projected dynamics \eqref{eq:genOU-SC-limit} in this setup, which reads 

\begin{align}\label{eq:lim-HamForm}
    dY_t = P(J-A)\Sigma^{-1}Y_t + \sqrt{2} PC dW_t, \quad \text{ where } P = \begin{pmatrix}
        0 & 0 \\ \alpha & I
    \end{pmatrix} \ \text{ and } \alpha= -K_{21} K_{11}^{-1}\,.
\end{align}
In the following we discuss the  projection $P$ that appears in the limiting dynamics above.  

\begin{rem}[Orthogonal and oblique projections.] We discuss the orthogonality of the projection $P$ for various choices of $K$.
First, observe that the projection $P$ is orthogonal with respect to the standard inner product if and only if $\alpha =0$ , i.e.\ $K_{21}=0$.

     On the other hand, if $A$ is positive definite, i.e.\ $A>0$, we can consider the weighted inner product $\langle x,y\rangle_A \coloneqq x^T A^{-1} y$ for $x,y \in \R^d$. In the weighted inner product space, the choice $K=A$ will result in an orthogonal projection, as we now show. The map $P$ is an orthogonal projection with respect to the inner product weighted by $A^{-1}$ if $\langle Pv - v, Pu \rangle_A = 0$ for any $u,v\in\R^d$. Writing $Pu = \left(0,w^2\right)^T$ and using $\alpha = - A_{11}^{-1}A_{12}$, we calculate
    \begin{align*}
        \langle Pv - v,  Pu \rangle_A &= \left(v^1\right)^T \left[ -(A_{11})^{-1} A_{12} (A^{-1})_{22} - (A^{-1})_{12}  \right] w^2 \,.
    \end{align*}
    By the expressions for block-matrix inversion (similar to \eqref{eq:inversesigma}) we have $-(A_{11})^{-1} A_{12} (A^{-1})_{22} - (A^{-1})_{12}=0$, i.e. $P$ is indeed orthogonal with respect to the inner product weighted by $A^{-1}$ if $K=A>0.$ For a similar discussion see~\cite[Remark 3]{SharmaZhang21}. \end{rem}

    A careful look at the asymptotic result in Theorem~\ref{thm:genOU-SC} reveals that the limit dynamics $Y_t$~\eqref{eq:lim-HamForm} is in fact the same as $PX_t$ with $K=0$, i.e\ the projection of the original unconstrained dynamics.

     The soft-constrained OU process~\eqref{eq:genOU-SC-Ham} can  explicitly be written as
    \begin{align*}
        \begin{pmatrix}dX^1_t \\ dX^2_t
        \end{pmatrix} = (J-A)\Sigma^{-1} \begin{pmatrix}
            X^1_t \\ X^2_t
        \end{pmatrix} dt -\frac{1}{\eps} 
        \begin{pmatrix}
            K_{11} X^1_t \\ K_{21} X^1_t
        \end{pmatrix} dt 
        + \sqrt{2}CdW_t,
    \end{align*}
    where we have used the explicit expression~\eqref{eq:genOU-ExplConst} for $\nabla|\xi|^2$. Choosing $K_{21}=0$ implies that the $X^2$ dynamics has no stiff terms (containing $\eps$) and follows the original unconstrained dynamics (i.e.\ with $K=0$). On the other hand, if $K_{21} \neq 0$ the $X^2$ dynamics is shifted by $\eps^{-1} K_{21} X^1_t$ and in the limit as $\eps \to 0$ this will result in an oblique projection with respect to the standard inner product.

The following result identifies general conditions under which the limit dynamics~\eqref{eq:lim-HamForm}, or equivalently~\eqref{eq:genOU-SC-limit-Expl}, admits $\mu_c$ as the correct invariant measure.
\begin{prop}\label{prop:genOU-SteSt}
Let $-K_{11} \in \R^{k \times k}$ be Hurwitz. Define  $\hat M \coloneqq \alpha M_{12} + M_{22} \in \R^{(d-k) \times (d-k)}$ where $\alpha=-K_{21}K_{11}^{-1}$ and $M=(J-A)\Sigma^{-1}$ (see Theorem~\ref{thm:genOU-SC}). Assume that $\hat M$ is Hurwitz and that $(\hat M,\hat C)$ is controllable, where $\hat C$ is defined in Theorem~\ref{thm:genOU-SC}. Then the (limiting) dynamics $(Y^1_t,Y^2_t)$~\eqref{eq:lim-HamForm}  admits the unique invariant measure $\mu^{\eps=0}\in\mathcal P(\R^d)$ given by 
\begin{equation}\label{eq:genOU-limSteSta}
    \mu^{\eps=0}(dy^1,dy^2) = \delta_{Y^1_0}(dy^1) \ \hat\mu^{\eps=0}(dy^2),  \ \ \  \text{where }  \ \ \mathcal P(\R^{d-k}) \ni \hat\mu^{\eps=0} = \mathcal N(\hat m,\hat \Sigma),
\end{equation}
and $\delta$ is the Dirac-delta measure. Here the mean $\hat m \in \R^{d-k}$ given by  
\begin{equation*}
    \hat m = - \hat M^{-1}(\alpha M_{11}+ M_{21}) Y^1_0
\end{equation*}
and the variance $\hat \Sigma \in \R^{(d-k)\times (d-k)}$ is the unique positive definite solution to 
\begin{equation}\label{eq:Steady-Lyap}
    \hat M\hat \Sigma + \hat \Sigma \hat M^T = -2\hat C \hat C^T.
\end{equation}

Moreover, $\hat\Sigma = \Sigma_c$ according to~\eqref{eq:genOU-condSteSta} if and only if the matrix
\begin{equation*}
    (\alpha+\Sigma_{21}\Sigma_{11}^{-1}) \left( (J_{11} + A_{11})\alpha^T + (J_{12} + A_{12}) \right) \in \R^{(d-k)\times (d-k)}
\end{equation*}
is skew-symmetric.
Assuming well-prepared initial datum for $Y^1_0$, i.e.\ $Y^1_0=b$, we have $\hat m = m_c$ if and only if $b\in \R^k$ is in the kernel of the matrix
\begin{equation*}
    \Sigma_{21}\Sigma_{11}^{-1} + \hat M^{-1}(\alpha M_{11}+M_{21}) \in \R^{(d-k)\times k}. 
\end{equation*}
 In particular, $\hat m =m_c$ for $b=0$. 
\end{prop} 

\begin{proof}
Since $Y^1_t \equiv Y^1_0 $ for the limiting dynamics~\eqref{eq:genOU-SC-limit-Expl}, the delta measure in $y^1$ follows. The invariant measure for $Y^2_t$ (for fixed value of $Y^1$) follows by using Proposition~\ref{prop:invmeasOU}. \\
Next we want to discuss conditions under which $\hat\Sigma=\Sigma_c$ for various choices of $K$.
First note that, using $A = CC^T$ as in~\eqref{eq:genOU-SC-Ham}, we have
\begin{equation} \label{eq:CC^T-form}
\begin{aligned}
   \hat C \hat C^T &= (\alpha C_{11} + C_{21}) (\alpha C_{11} + C_{21})^T   + (\alpha C_{21} + C_{22}) (\alpha C_{21} + C_{22})^T \\
    &=  \alpha A_{11} \alpha^T + \alpha A_{12} + A_{21} \alpha^T + A_{22} = (PAP^T)_{22}.
\end{aligned}
\end{equation}
This requires a study of the Lyapunov equation~\eqref{eq:Steady-Lyap} which we now compute explicitly. Using block-matrix inversion, 
\begin{equation}
    \Sigma^{-1}=\begin{pmatrix}
\Sigma_{11} & \Sigma_{12} \\
\Sigma_{21} & \Sigma_{22}    
\end{pmatrix}^{-1}
        = \begin{pmatrix}
            \Sigma_{11}^{-1} - \Sigma_{11}^{-1} \Sigma_{12} \Sigma_c^{-1} \Sigma_{21} \Sigma_{11}^{-1}   
            & -\Sigma_{11}^{-1}\Sigma_{12}\Sigma_c^{-1}\\
            -\Sigma_c^{-1}\Sigma_{21}\Sigma_{11}^{-1}
            & \Sigma_c^{-1}
        \end{pmatrix} \label{eq:inversesigma}
\end{equation}
where $\Sigma_c=\Sigma_{22}-\Sigma_{21}\Sigma_{11}^{-1}\Sigma_{12}$ and we have used $\Sigma^T_{12}=\Sigma_{21}$. Using the definition of $M$, it follows that 
\begin{align*}
    M_{12} &=(J_{11}-A_{11})(\Sigma^{-1})_{12} + (J_{12}-A_{12})(\Sigma^{-1})_{22},\\
    M_{22} &=(J_{21}-A_{21})(\Sigma^{-1})_{12} + (J_{22}-A_{22})(\Sigma^{-1})_{22},
\end{align*}
and therefore, using the explicit form of $\Sigma^{-1}$ in \eqref{eq:inversesigma}, we can expand $\hat M=\alpha M_{12} + M_{22}$ to arrive at  
\begin{equation*}
\hat M = \Bigl[  -\alpha (J_{11}-A_{11}) \Sigma_{11}^{-1}\Sigma_{12} + \alpha(J_{12}-A_{12}) - (J_{21}-A_{21})\Sigma_{11}^{-1}\Sigma_{12} + (J_{22}-A_{22})
\Bigr] \Sigma_c^{-1}.
\end{equation*}
Since we are interested in showing that  $\hat \Sigma=\Sigma_c$, we need to show that the Lyapunov equation~\eqref{eq:Steady-Lyap} holds, i.e.\  
\begin{equation}\label{eq:mod-Lyap}
2\hat C\hat C^T  + \hat M\Sigma_c+ \Sigma_c \hat M^T = 0\,. 
\end{equation}
Using the explicit formulae above and~\eqref{eq:CC^T-form} we have
\begin{equation}\label{eq:mod-Lyap-LHS}
\begin{aligned}
    2&\hat C\hat C^T +\hat M\Sigma_c+ \Sigma_c \hat M^T \\ 
    &= 2\alpha A_{11} \alpha^T + 2\alpha A_{12} + 2 A_{21} \alpha^T + 2 A_{22} \\
    & \ \  -\alpha (J_{11}-A_{11})\Sigma_{11}^{-1}\Sigma_{12} + \alpha(J_{12}-A_{12}) - (J_{21}-A_{21}) \Sigma^{-1}_{11}\Sigma_{12} + (J_{22}-A_{22}) \\
    & \ \ - \Sigma_{21}\Sigma_{11}^{-1} (-J_{11}-A_{11})\alpha^T + (-J_{21}-A_{21})\alpha^T -\Sigma_{21}\Sigma_{11}^{-1} (-J_{12}-A_{12}) + (-J_{22}-A_{22})\\
    & = 2\alpha A_{11} \alpha^T  - \alpha(J_{11}-A_{11})\Sigma_{11}^{-1}\Sigma_{12} + \alpha(J_{12}+ A_{12}) 
    + \Sigma_{21}\Sigma_{11}^{-1} (J_{11}+A_{11})\alpha^T - (J_{21}-A_{21})\alpha^T \\
    & \ \ -(J_{21}-A_{21})\Sigma_{11}^{-1}\Sigma_{12} + \Sigma_{21}\Sigma_{11}^{-1}(J_{12}+A_{12})\\
    & = (\alpha+\Sigma_{21}\Sigma_{11}^{-1})(J_{11}+A_{11})\alpha^T + \Bigl( [\alpha+\Sigma_{21}\Sigma_{11}^{-1}](J_{11}+A_{11})\alpha^T \Bigr)^T\\
     &\quad +(\alpha+\Sigma_{21}\Sigma_{11}^{-1})(J_{12}+A_{12}) + \Bigl( (\alpha+\Sigma_{21}\Sigma_{11}^{-1})(J_{12}+A_{12}) \Bigr)^T \\
     &= \underbrace{(\alpha+\Sigma_{21}\Sigma_{11}^{-1}) \left( (J_{11} + A_{11})\alpha^T + (J_{12} + A_{12}) \right)}_{= R} + \underbrace{\left((\alpha+\Sigma_{21}\Sigma_{11}^{-1}) \left( (J_{11} + A_{11})\alpha^T + (J_{12} + A_{12}) \right) \right)^T}_{=R^T}\,,
\end{aligned}
\end{equation}
where the second equality follows since $A=A^T$ and $J=-J^T$ which implies that $J_{11}^T=-J_{11}$, $J_{12}^T=-J_{21}$, $J_{22}^T=-J_{22}$, and 
    \begin{align*}
         2\alpha A_{11} \alpha^T &= \alpha (A_{11}+J_{11})\alpha^T + \alpha(A_{11}-J_{11})\alpha^T .
    \end{align*}
The Lyapunov equation \eqref{eq:mod-Lyap} is thus satisfied if and only if the last line in \eqref{eq:mod-Lyap-LHS} equates to zero.  In case $\alpha\neq 0$, this requires $R \in \R^{(d-k) \times (d-k)}$ to be skew symmetric. 
The condition for $\hat m=m_c$ follows directly from the definitions of $\hat m$ and $m_c$. 
\end{proof}

\subsection{Confinement mechanism}
\label{ssec:MatrixK}

The general conditions outlined in the result above do not provide intuition about the structure of $K$, which is required to ensure that the limiting invariant measure $\hat\mu^{\eps=0}$\eqref{eq:genOU-limSteSta} matches the conditional distribution $\mu_c$~\eqref{eq:genOU-condSteSta}. However, there are some natural choices for $K$, such as $K=A,A-J$, which appear in the literature. 

Let us first consider the choice $K=A$. Simple choices  of the matrices $A,J,\,\Sigma$ lead to $\hat\mu^{\eps=0} \neq \mu_c$. For instance, if 
\[A=\Sigma=I\quad\textrm{and} \quad J=\begin{pmatrix} 0 & -I \\ I & 0 \end{pmatrix}\,,
\]
i.e.\ the invariant measure for the unconstrained OU process is $\mu=\mathcal N(0,I)$, then $b =\hat m \neq m_c=0$ for any $b\neq 0$.\footnote{To be precise we also need that the original and coarse-grained space are even dimensional, i.e. $d,k$ are even naturals.} 
Alternatively, for $\beta>1$ and 
\[
A=I\,,\quad \Sigma=\begin{pmatrix}I & I \\ I & \beta I \end{pmatrix}\,,\quad J=\begin{pmatrix} 0 & J_{12} \\ -J_{12}^T & 0 \end{pmatrix}\,,
\] 
with $J_{12}\neq -J^T_{12}$, we have that $\hat\Sigma\neq \Sigma_c$.

The result below discusses other admissible choices for $K$. Note that the choice $K=-(J-A)$ is expected to work since in this special case the softly constrained OU process~\eqref{eq:genOU-SC-Ham} admits the unique invariant measure
\begin{equation*}
    \mu^\eps(dx) = Z^{-1}\exp\biggl(-\frac12 x^T \Sigma^{-1}x - \frac{1}{2\eps}|x^1-b|^2 \biggr),
\end{equation*}
where $Z=Z^\eps$ is a normalisation constant, which is expected to converge to 0 as $\eps\to 0$ since the (unnormalised) Gaussian density becomes singular. This is made precise in the following result. 

\begin{prop}\label{prop:genOU-choiceK}
Under the assumptions in Proposition~\ref{prop:genOU-SteSt} the following holds.
\begin{enumerate}[label=(\roman*)]  
    \item\label{item:K=A-J} If $K=A-J$ then $\hat\mu^{\eps=0}=\mu_c$. In particular if $J=0$ and $K=A$ then $\hat\mu^{\eps=0}=\mu_c$.  
    \item\label{item:K=Sigma} If $K=\Sigma$ then $\hat\Sigma=\Sigma_c$. In general, $\hat m\neq m_c$ for $b\neq 0$. 
\end{enumerate}
\end{prop}
\begin{rem}
    Let us mention that in Proposition \ref{prop:genOU-choiceK} it is enough to assume that either $\hat M$ is Hurwitz or $(\hat C, \hat M)$ is controllable, as the other assumption is implied. This follows because we  have a positive definite solution to the Lyapunov equation given by $\Sigma_c$.
\end{rem}
\begin{proof}
Recall from the proof of Proposition~\ref{prop:genOU-SteSt} that $\hat\Sigma=\Sigma_c$ if and only if $R=-R^T,$ where $R \in \R^{(d-k)\times (d-k)}$ is defined in~\eqref{eq:mod-Lyap-LHS}.
The simplest choice would be $R=0$. Since $R$ is given by the product of two factors, we can choose $\alpha$ such that one factor vanishes to yield $R=0$; this corresponds to two the following choices
\begin{enumerate}[leftmargin=*,label=(\roman*)]
\item $K=A-J$ which results in $\alpha=-(J_{21}-A_{21})(J_{11}-A_{11})^{-1}$,
\item $K=\Sigma$ which results in $\alpha=-\Sigma_{21}\Sigma_{11}^{-1}$.
\end{enumerate}    
Note that the previous asymptotic results only hold if $-K_{11}$ is Hurwitz, which guarantees that $J_{11}-A_{11}$ is invertible when $K=A-J$.    
Recalling Proposition~\ref{prop:genOU-SteSt}, we shall now compare the means 
\begin{equation*}
 m_c=\Sigma_{21} \Sigma_{11}^{-1}b \ \text{ and } \ \hat m = -\hat M^{-1}(\alpha M_{11} + M_{21})b.
\end{equation*}
For $K=A-J$, using the explicit expressions for $\Sigma^{-1}$ in \eqref{eq:inversesigma} we have
     \begin{align*}
         \hat M &= \left[(A_{21} - J_{21})(A_{11} - J_{11})^{-1} (A_{12} - J_{12}) - (A_{22} - J_{22}) \right] \Sigma^{-1}_c\,, \\
          \alpha M_{11} + M_{21} &= -\left[(A_{21} - J_{21})(A_{11} - J_{11})^{-1} (A_{12} - J_{12}) - (A_{22} - J_{22}) \right] \Sigma^{-1}_c\Sigma_{21}\Sigma_{11}^{-1} 
     \end{align*}
and hence $\hat m= \Sigma_{21} \Sigma_{11}^{-1} b =m_c$. In summary, $\hat\mu^{\eps=0}=\mu_c$ if $K=A-J$.

For $K=\Sigma$, choosing $k,d$ even with $2k\leq d$, $b\neq 0$, $A=\Sigma=I$ and $J=-J^T$ the canonical symplectic matrix, we have $b =\hat m \neq m_c=0$. Therefore, in general, for the case $K=\Sigma$, $\hat\mu^{\xi=0}$ and $\mu_c$ need not agree.
\end{proof}

We briefly discuss the reversible case as an example, i.e. $A=A^T>0$,  $J=0$, and a soft constraint with $K=I$. We show that the projected dynamics does not necessarily leave $\mu_c$ invariant, which means that it is the structure of the noise rather than the reversibility or irreversibility of the dynamics that determines whether the invariant measure is robust under constraining or not. 

\begin{rem}[Reversible dynamics]
    Consider the case $J=0$ and $K=I$, such that $\alpha=-K_{21}K_{11}^{-1}=0$. For $\hat\Sigma=\Sigma_c$, Proposition \ref{prop:genOU-SteSt} states that the matrix 
$\Sigma_{21}\Sigma_{11}^{-1}A_{12}$ needs to be skew-symmetric, i.e.\
\begin{equation*}
    \Sigma_{21}\Sigma_{11}^{-1}A_{12} + A_{21}\Sigma_{11}^{-1}\Sigma_{12} = 0.
\end{equation*}
This is true if $A_{12}=0=A_{21}$ or $\Sigma_{12}=0=\Sigma_{21}$, in which case even $\mu^{\xi=0} = \mu_c$. An example, for which  $\mu^{\xi=0} \neq \mu_c$ is given by the matrices 
\[
A=\begin{pmatrix}
    I & I \\ I & \beta I
\end{pmatrix}\quad\textrm{and}\quad \Sigma=\begin{pmatrix}
    I & I \\ I & \theta I
\end{pmatrix}\,,
\]
with $\beta,\theta>1$. The matrices $A$ and $\Sigma$ are symmetric positive definite, and in the language of Markov chain Monte Carlo  algorithms (e.g.~\cite{girolami2011riemann}), the associated OU process is a preconditoned version of 
\[
d\tilde{X}_t = -\Sigma^{-1} \tilde{X}_t\,dt + \sqrt{2}\,dW_t\,,
\]
with uncorrelated noise. By construction, $\tilde{X}$ and the preconditioned system,
\[
dX_t = -A\Sigma^{-1} X_t\,dt + \sqrt{2A}\,dW_t\,,
\]
with correlated noise, have the same invariant measure $\mu=\cN(0,\Sigma)$, but the expression above gives
\begin{equation*}
    \Sigma_{21}\Sigma_{11}^{-1}A_{12} + A_{21}\Sigma_{11}^{-1}\Sigma_{12} = 2I \neq 0,
\end{equation*}
which entails $\mu^{\xi=0} \neq \mu_c$. Note that if instead we had chosen $A=I=K$ and $J=0$ we would have been back to the case \ref{item:K=A-J} and 
 $\mu^{\xi=0} = \mu_c$.
\end{rem}

The proof of Proposition~\ref{prop:genOU-choiceK} suggests that there are infinitely many choices for $K$ which lead to $\hat\mu^{\eps=0}=\mu_{c}$ which is made precise in the following result. 

\begin{prop}
Under the assumptions in Proposition~\ref{prop:genOU-SteSt} the following holds.
\begin{enumerate}[label=(\roman*)]  
\item \label{case:infKA-J} If $K=\begin{pmatrix} K_{11} & * \\ (A-J)_{21}(A-J)_{11}^{-1}K_{11} & * \end{pmatrix}$ with $(A-J)_{11}$  invertible, then $\hat\mu^{\eps=0}=\mu_c$.
\item\label{case:infKSigma} If $K=\begin{pmatrix}K_{11} & * \\ \Sigma_{21}\Sigma_{11}^{-1}K_{11} & * \end{pmatrix}$ then $\hat\Sigma=\Sigma_c$. In general, $\hat m\neq m_c$ for $b\neq 0$. 
\item\label{case:infKSpec} If $\Sigma_{12}=0$ then $K=\begin{pmatrix} K_{11} & * \\ (A_{21}-J_{21}) (A_{11} + \bar J )^{-1}K_{11} & * \end{pmatrix}$ where  $\bar J \in \R^{k \times k}$ is any skew symmetric matrix such that $(A_{11} + \bar J)$ is invertible, then $\hat\Sigma=\Sigma_c$.
\end{enumerate}
\end{prop}

Before we prove this proposition let us give some intuition about its meaning. It provides us with infinitely many choices for the matrix $K$ in the soft-constrained problem \eqref{eq:genOU-SC-Ham}, so that we sample the correct target conditional measure. 
To be more precise, the choice of $-K_{11}$  is free as long as it is Hurwitz. Vividly speaking this means that $X^1$ can approach the constraint manifold $\xi^{-1}(0)$ in any way, e.g. via a gradient flow ($K_{11}=I$) or spiralling towards it (e.g. $K_{11} = I + J$, where $J=-J^T\neq 0$).  
On the other hand the choice for $K_{21}$ is not free at all and is dictated by the first $k$ columns of the matrices $A, J$ (case \ref{case:infKA-J})and $\Sigma$ (case \ref{case:infKSigma}) as well as $K_{11}$. This special form of $K_{21}$ guarantees that $X^1$ enters the $X^2$ dynamics in the exact same way as in Proposition \ref{prop:genOU-choiceK}.
Case~\ref{case:infKSpec} tells us that if the invariant measure of the unconstrained process~\eqref{eq:genOU-SC-Ham} with $K=0$ is uncorrelated in the first $k$ and remaining $d-k$ dimensions, we can freely choose $K_{21}$ via the skew symmetric matrix $\bar J \in \R^{k \times k}.$

\begin{proof}
The proof of the first two cases is as in the proof of Proposition~\ref{prop:genOU-choiceK}, realizing that $\alpha$ is unchanged.

If $\Sigma_{12}=0$, then for any $\bar J = -\bar J^T \in \R^{d\times d}$ the Lyapunov equation~\eqref{eq:mod-Lyap} can be written as (cf.~\eqref{eq:mod-Lyap-LHS})
\begin{align*}
    2 &\alpha A_{11} \alpha^T + \alpha(J_{12} + A_{12}) + (A_{21} - J_{21})\alpha^T 
     \\ &= \alpha \bigl[(\bar J + A_{11}) \alpha^T + (J_{12} + A_{12})\bigr] + \left(\alpha \bigl[(\bar J + A_{11}) \alpha^T + (J_{12} + A_{12})\bigr] \right)^T =0 \,, 
\end{align*}
where we have used $\bar J = - \bar J^T \in \R^{d \times d}$, which gives
\[2\alpha A_{11}\alpha^T = \alpha (A_{11} + \bar J)\alpha^T + \alpha (A_{11} - \bar J)\alpha^T.\]

Hence, it is easy to check that for any skew symmetric $\bar J $ such that $A_{11} + \bar J$ is invertible, $\alpha = -(A_{21}- J_{21})(A_{11} + \bar J)^{-1}$ satisfies \eqref{eq:mod-Lyap} and hence leads to the correct covariance $\hat \Sigma = \Sigma_c$.  Equivalently, for a given Hurwitz matrix $K_{11} \in \R^{d \times d}$, choose 
\[
K=\begin{pmatrix}
    K_{11} & * \\ (A-J)_{21} (A-\bar J)_{11}^{-1}K_{11} & * \end{pmatrix}\,.
\]
\end{proof}

\section{Numerical illustration}
 \label{sec:numEx}

We consider two illustrative examples that demonstrate (a) the uniform  pathwise convergence of the softly contrained process on any time interval $[\delta, T]$ for some small $\delta>0$ and (b) the preservation of the invariant measure under hard and soft constraints for a suitably chosen matrix $K$.

\subsection{Particle coupled to a heat bath}

Consider the $(1+2n)$-dimensional system
\begin{equation}\label{heatbath}
    dX_t = (J - A)X_t\,dt + \sqrt{2}C\,dW_t\,,
\end{equation}
with $W=(V,U)$ denoting $(1+n)$-dimensional Brownian motion and 
\[
J-A = \begin{pmatrix}
   -L & 0 & \lambda^T\\0 & 0 & I\\-\lambda & -I & -\gamma
\end{pmatrix}\in\R^{(1+2n)\times (1+2n)}\,,\quad 
C = \begin{pmatrix}
    \sqrt{L} & 0\\ 0 & 0 \\ 0 & \sqrt{\gamma}
\end{pmatrix}\in\R^{(1+2n)\times (1+n)}\,.
\]
where $L>0$, $\gamma\in\R^{n\times n}$ symmetric positive definite, and $\lambda\in\R^{n}$. It can be readily seen that under these assumptions, the matrix $J-A$ is stable, and the pair $(J-A,C)$ is completely controllable. 
In what follows we use the shorthands $X^1=\zeta\in\R$ and $X^2=(q,p)\in\R^{2n}$, and we consider the codimension 1 constraint 
\[
\zeta = 0\,.
\]
The softly constrained system now reads
\begin{equation}\label{heatbath_eps}
    dX_t = (J - A)X_t\,dt - \frac{1}{\eps}K X_t\,dt + \sqrt{2}C\,dW_t\,,
\end{equation}
with 
\[
K = \begin{pmatrix}
    1 & 0 & 0\\0 & 0 & 0\\ 0 & 0 & 0 
    \end{pmatrix}\in\R^{(1+2n)\times (1+2n)}\,.
\]

Following Theorem \ref{thm:genOU-SC}, the limit system can be recast as
\begin{equation}\label{heatbath_lim}
    dY^2_t = (\bar{J} - \bar{A})Y^2_t dt + \sqrt{2}\bar{C}dU_t
\end{equation}
where $Y^2$ denotes the limit of the unconstrained component $X^2=(q,p)$, and $U$ denotes standard $n$-dimensional Brownian motion. The coefficients of the constrained system read
\[
\bar{J} - \bar{A} = \begin{pmatrix}
    0 & I\\ -I & -\gamma 
\end{pmatrix}\in\R^{2n\times 2n}\,,
\quad \bar C = \begin{pmatrix}
    0\\ \sqrt{\gamma}
\end{pmatrix}\in\R^{2n\times n}\,.
\]

\begin{figure}
    \centering
    \includegraphics[width=0.495\linewidth]{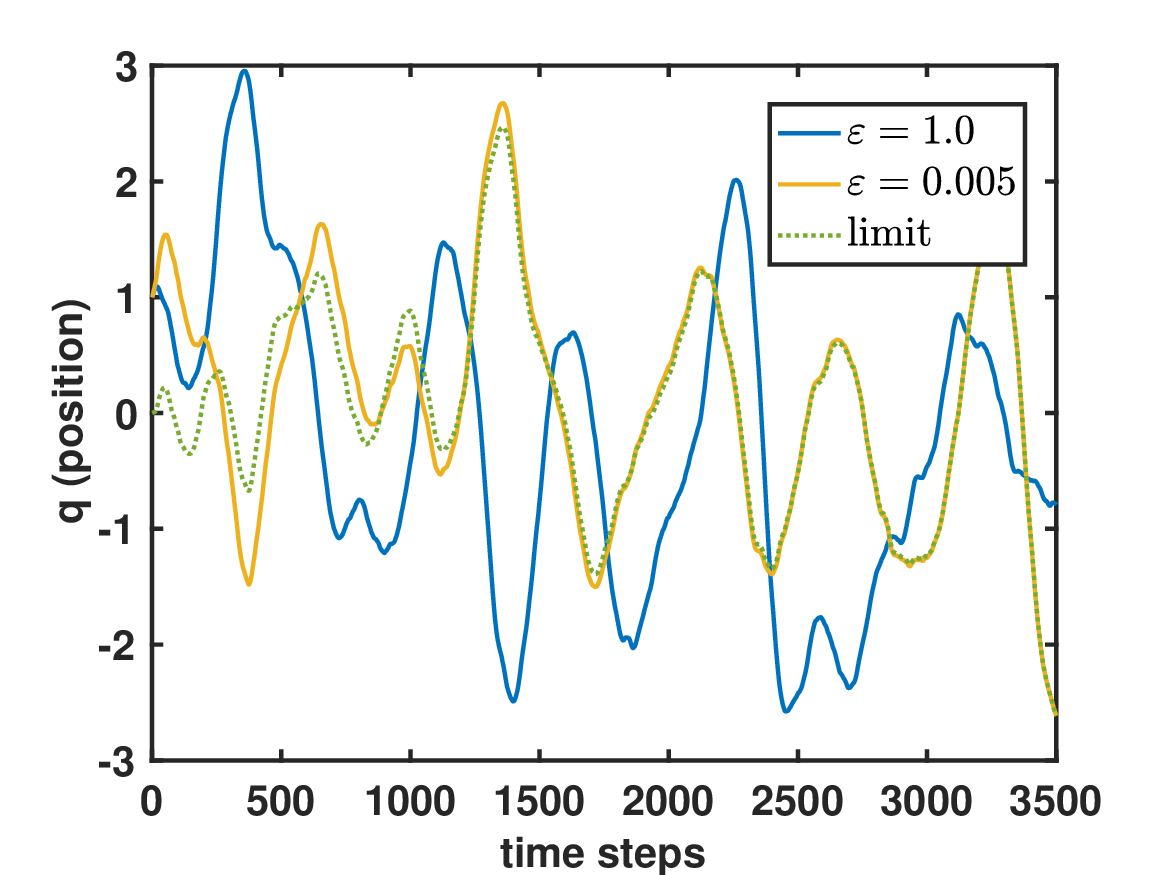}
    \includegraphics[width=0.495\linewidth]{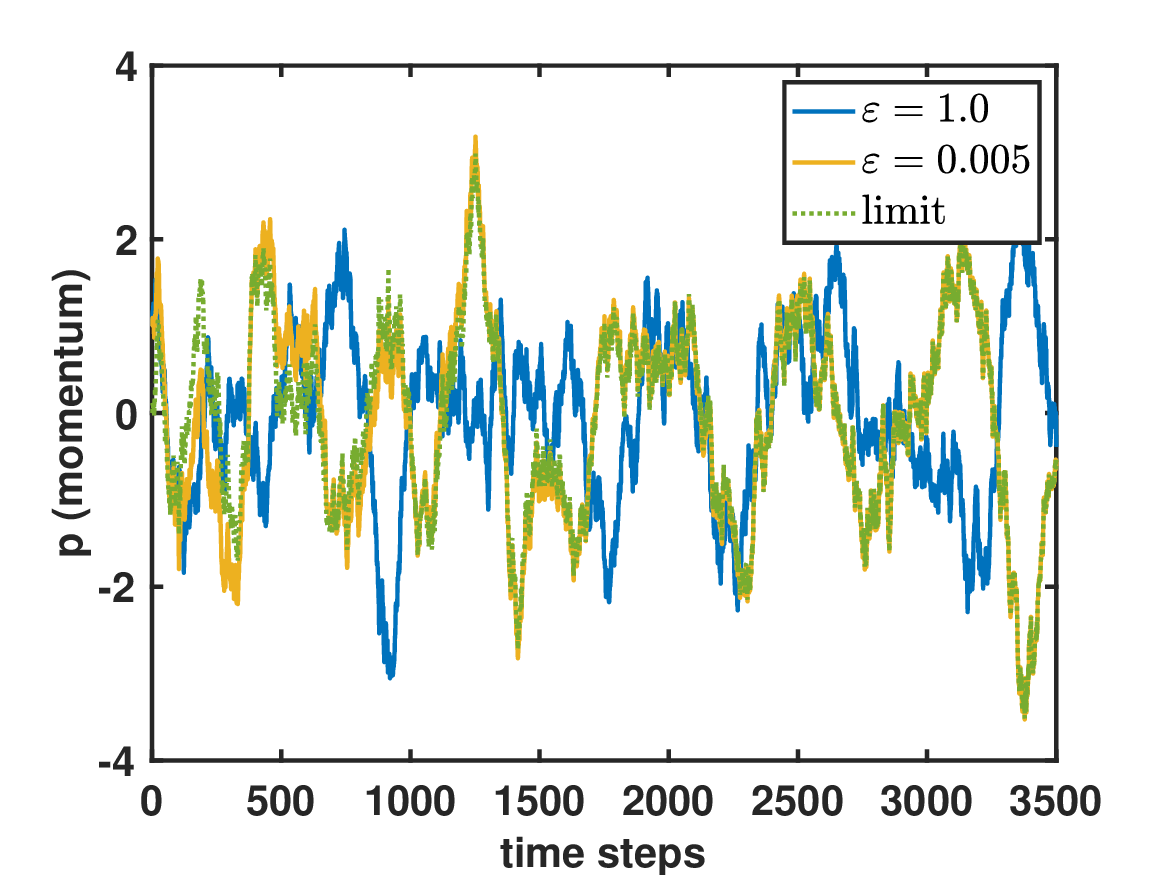}
    \caption{Typical realisations of the softly constrained  heat bath model (\ref{heatbath_eps}) for $n=1$ and its limit (\ref{heatbath_lim}) for the constrained thermostat variable $X^1=0$. The left panel shows the position variable $q$, the right panel shows the momentum variable $p$.}
    \label{fig:heatbath1}
\end{figure}

Figure \ref{fig:heatbath1} shows typical realisations of the softly constrained 3-dimensional heat bath model (\ref{heatbath_eps}) for $\eps=1$ and $\eps=0.005$. In the simulation, the initial condition of the constrained variable $X^1=\zeta$ has been set to $X^1_0\neq 0$, which is inconsistent with the requirement $\zeta=0$. Both panels, for the position and variable $q$ (left) and the momentum variable $p$ (right) display the transient initial layer phenomenon for inconsistent initial conditions that is in accordance with Remark \ref{rem:genOU-IC}.


For $\gamma>0$, the softly constrained dynamics has a unique Gaussian invariant measure with positive definite covariance matrix $\Sigma_\eps$  for all $\eps>0$. Even for $d=3$, the covariance matrix has a complicated expression for $\eps>0$, with correlations between the thermostat variable $X^1=\zeta$ and position and momentum variables. Yet, the limit $\eps\to 0$ is consistent with the constrained dynamics, in that
\[
\lim_{\eps\to 0}\Sigma_\eps = \begin{pmatrix}
    0 & 0 & 0\\0 & 1 & 0\\ 0 & 0 & 1 
    \end{pmatrix}\,.
\]
(Here and in what follows, we confine our attention the 3-dimensional case.) As a consequence, the approximation of the time marginal $\mathrm{law}(q_t,p_t)$ by the constrained dynamics, $Y_t$ is uniform in time.

\subsubsection*{Deterministic limit case $\gamma=0$ and nonuniform-in-time approximation}

The situation changes when $\gamma=0$, where we confine the following discussion to the case $n=1$. The matrix $J-A$ is  Hurwitz and so is $J-A- \frac{1}{\eps}K$ for every $\eps>0$ (cf.~\cite{maddocks1995stability}), moreover the pair $(J-A,C)$ is completely controllable. The softly constrained dynamics has a unique zero-mean Gaussian invariant measure $\rho^\eps$ with $3\times 3$ covariance matrix 
\[
\tilde{\Sigma}_\eps=\frac{\eps}{1+\eps} I
\]
that converges to the zero matrix as $\eps\to 0$. As a consequence, the finite time marginal of the process converges in distribution to a Dirac centred at $X=0$, in other words: 
\[
\rho^\eps \stackrel{*}{\rightharpoonup} \delta_0\quad\textrm{as}\quad  \eps\to 0\,.
\]
On the other hand, the dynamics (\ref{heatbath}) for $\gamma=0$ and subject to the constraint $X^1=0$ is the deterministic Hamiltonian system for $X^2=(q,p)$:
\begin{equation}\label{hamiltonODE}
\begin{aligned}
    \dot{q} & = p\\ \dot{p} & = -q\,.
\end{aligned}
\end{equation}
The dynamics has infinitely many invariant measures, and since (\ref{hamiltonODE}) conserves the energy
\[
H(q,p) = \frac{1}{2}\left(q^2 + p^2\right)\,,
\]
none of these invariant measures of (\ref{hamiltonODE}) agrees with the weak--* limit of $\rho^\eps$, unless $q_0=p_0=0$. Nevertheless, assuming consistent initial conditions for $X^1$, the softly constrained dynamics converges to (\ref{hamiltonODE}) in a pathwise sense, uniformly on any compact time interval $[0,T]$. While this is in accordance with Theorem \ref{thm:genOU-SC}, it implies that the approximation cannot be uniform in time if $\gamma=0$.

Indeed, as the left panel of Figure \ref{fig:heatbath3} shows, the soft constraint dynamics for initial conditions $X^1=0$ and $X^2=(1,1)$ converges uniformly on $[0,T]$ to the solution of the Hamiltonian system (\ref{hamiltonODE}), which is a periodic motion on a circle of radius $\sqrt{2}$. Conversely, it is demonstrated in the right panel of the figure that the long term dynamics for small $\eps$ is dissipative and spirals into the unique fixed point at $(q,p)=(0,0)$. The numerical integration of the stochastic dynamics has been carried out with an Euler-Maruyama scheme with step size $h=10^{-4}$, whereas the deterministic dynamics has been discretised using a symplectic Euler method that guarantees long-term stability and approximate energy conservation. 
\begin{figure}
    \centering
    \includegraphics[width=0.495\linewidth]{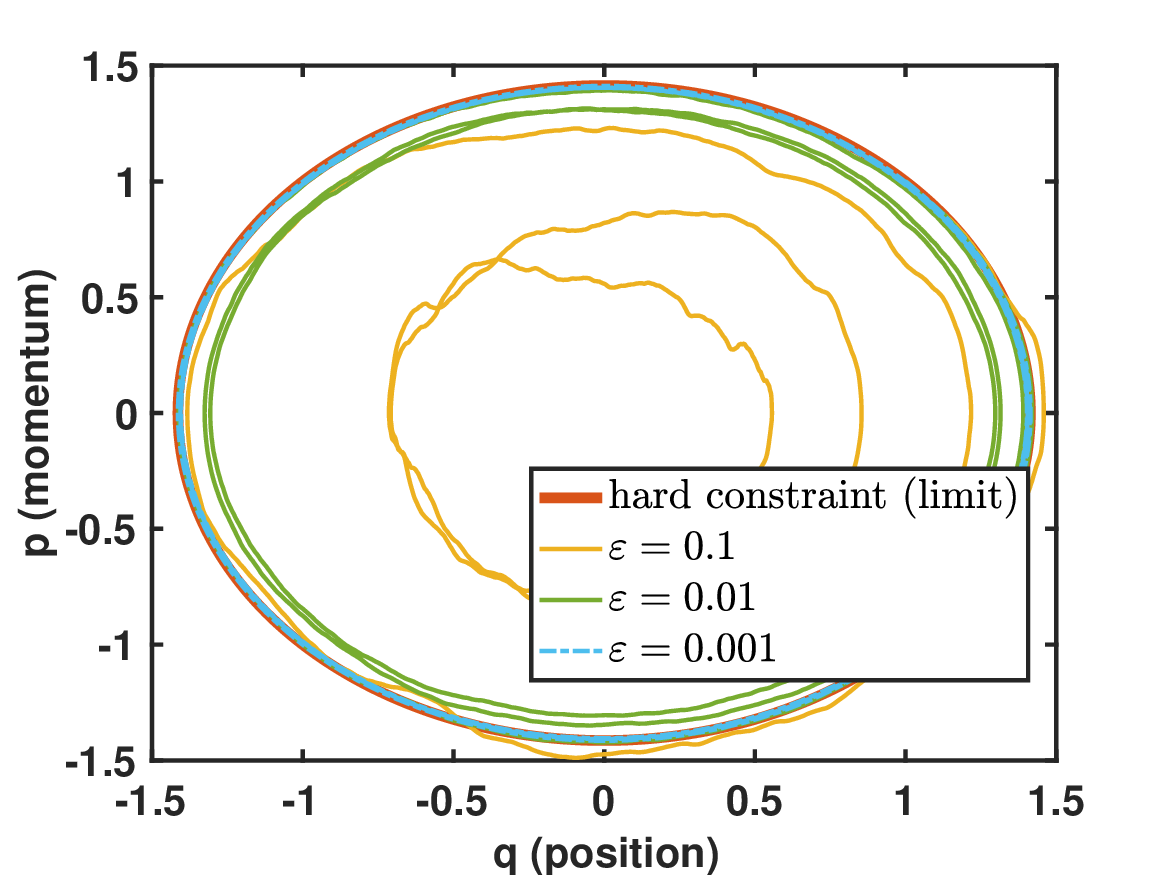}
    \includegraphics[width=0.495\linewidth]{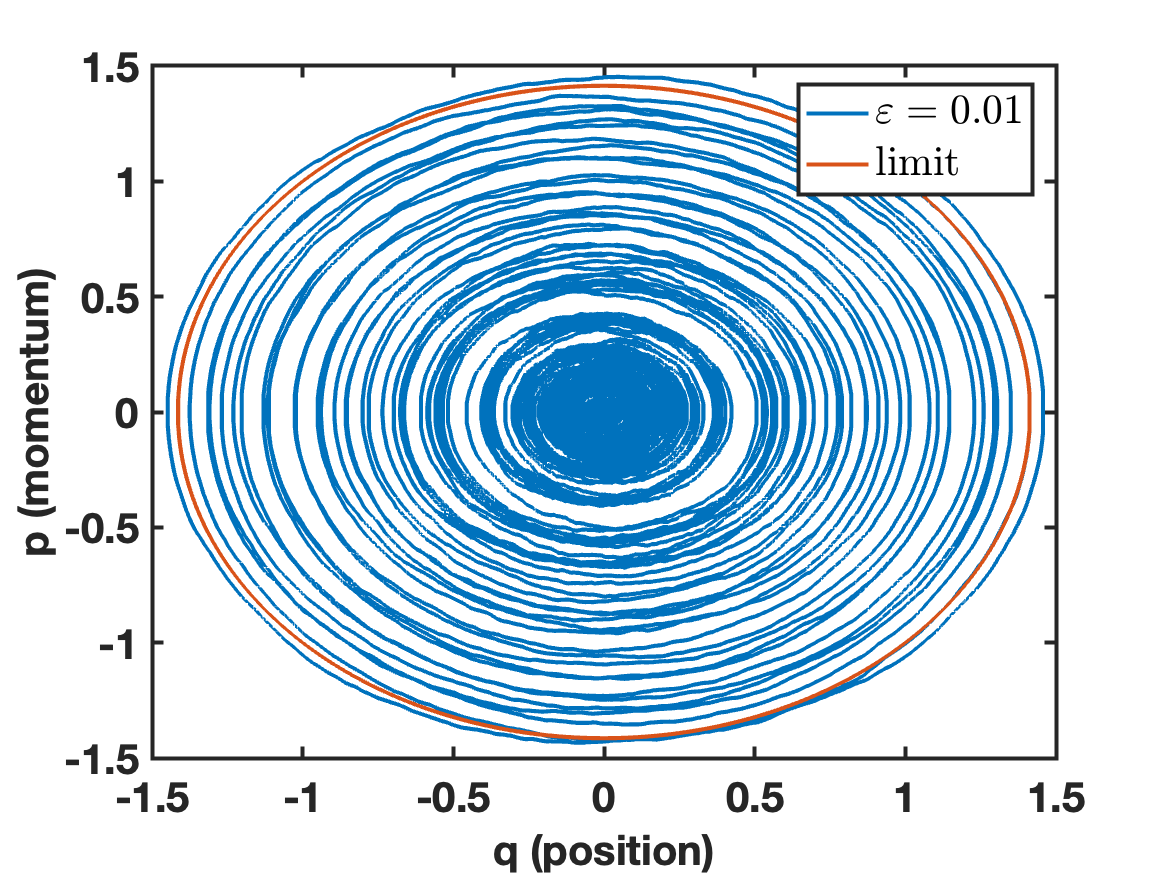}
    \caption{On any bounded time interval, the softly constrained dynamics for $\gamma=0$ converges pathwise to the the limit dynamics (left panel), while the long term dynamics for small $\eps$ departs from the deterministic Hamiltonian limit dynamics and spirals inwards towards the limit invariant measure $\delta_0$ (right panel).}
    \label{fig:heatbath3}
\end{figure}

It is worth mentioning, that if one makes the correct choice for $K$ according to Proposition \ref{prop:genOU-choiceK}, e.g. 
\begin{align*}
    K = \begin{pmatrix}
        L & 0 & 0 \\
        0 & 0 & 0 \\
        \lambda & 0 & 0
    \end{pmatrix},
\end{align*} 
then the softly constrained dynamics admits the unique invariant measure $\rho^\eps=\mathcal{N}(0,\tilde{\Sigma}_\eps)$, with covariance
\[
\tilde{\Sigma}_\eps = \begin{pmatrix}
    \frac{\eps}{1+\eps} & 0 & 0\\0 & 1 & 0\\ 0 & 0 & 1 
    \end{pmatrix}\in\R^{3\times 3}\,. 
\]
For every $\eps>0$, the invariant measure $\rho^\eps$ is approached at an exponential rate that is determined by the real part of the principal eigenvalue of the system matrix
\begin{align*}
    (J-A)I - \frac{1}{\eps} K = \begin{pmatrix}
        -L(1+\frac{1}{\eps}) & 0 & \lambda^T \\
        0 & 0 & I \\
        -\lambda(1+\frac{1}{\eps}) & -I & 0
    \end{pmatrix},
\end{align*} 
(For the sake of transferrability of the heat bath model, we adopt the matrix notation from the multidimensional case $n>1$, even though $\lambda^T=\lambda\in\R$ for $n=1$.) While one eigenvalue that corresponds to the constrained direction diverges as $\mathcal{O}(-\eps^{-1})$, the matrix has another pair of eigenvalues given by  
\[
\nu_{1,2}\simeq -\frac{\lambda^2}{2}\pm\sqrt{\frac{\lambda^2}{4}-L}\quad\textrm{as}\quad \eps\to 0\,.
\] 
For example, for $\lambda=1$ and $L=1$, the spectral abscissa converges to $-1/2$. As a consequence, the (degenerate) Gaussian limit measure 
\[
\rho^0 = \delta_{z=0}\otimes \mathcal{N}(0,I_{2\times 2})\,,
\]
which is the weak--* limit of $\rho^\eps$ as $\eps\to 0$, is approached at a finite exponential rate. This is consistent with the limit dynamics that is given by the projected equation 
\begin{equation}\label{eq:proj-GLE}
        dY^2_t = (\bar{J}-\bar{A})Y^2_t dt + \sqrt{2}\bar{C} dW_t 
\end{equation}
for $Y^2=(q,p)$ that is derived from the oblique projection
\[
  P = \begin{pmatrix}
        0 & 0 & 0 \\
        0 & I & 0 \\
        - \lambda L^{-1}& 0 & I
    \end{pmatrix},
\]
with the resulting drift and diffusion coefficients (neglecting zero rows and columns)
\begin{align*}
    (\bar{J}-\bar{A}) = (P(J-A))_{2 2} =  \begin{pmatrix}
        0 & I \\ -I & - \lambda L^{-1}\lambda^T
    \end{pmatrix}\,, \quad \bar{C} = (PC)_2 = \begin{pmatrix}
        0 & \\ - \lambda L^{-1/2} & 
    \end{pmatrix}.
\end{align*}

As a consequence we recognise (\ref{eq:proj-GLE}) as an underdamped Langevin equation with unique invariant measure, given by the $(q,p)$-marginal of $\rho^0$. The negative sign in front of the diffusion coefficient is irrelevant for convergence of the law of paths, but it guarantees a pathwise approximation that is uniform in time.   

The convergence of the softly constrained dynamics to (\ref{eq:proj-GLE}) for all $t>0$ is confirmed by the numerical simulations shown in Figure \ref{fig:heatbath4}; the left panel shows a typical realisation for different values of $\eps$. For $\eps\to 0$ the realisation converges almost surely and uniformly on $[0,T]$. On the other hand, the long term dynamics correctly reproduces the (unnormalised) marginal density $\exp(-H)$ as the right panel of the figure shows.

\begin{figure}
    \centering
    \includegraphics[width=0.495\linewidth]{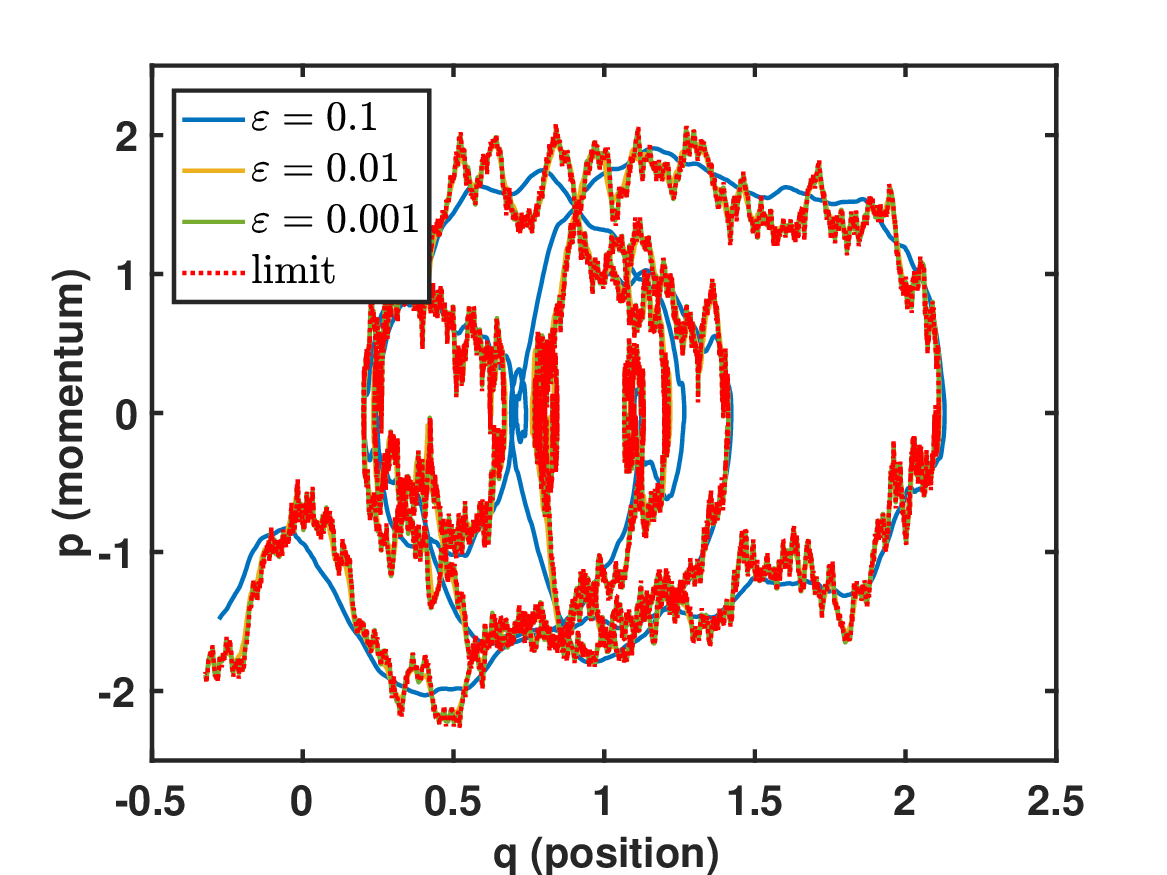}
    \includegraphics[width=0.495\linewidth]{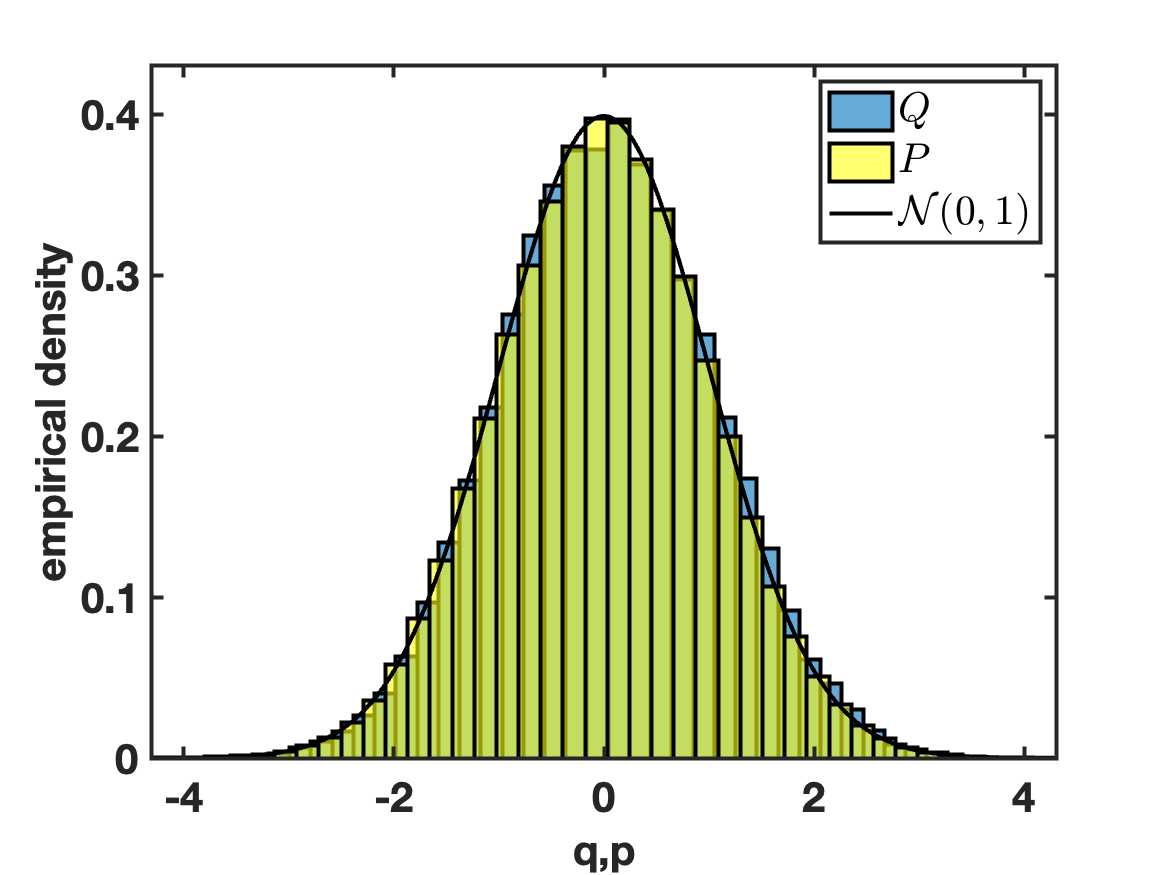}
    \caption{Left panel: typical realisations for $\eps\in\{0.1,\,0.01,\,0.001\}$ and the limiting underdamped Langevin dynamics (\ref{eq:proj-GLE}), with parameters $\lambda=1$ and $L=1$. Right panel: the long term dynamics correctly reproduces the (unnormalised) standard Gaussian marginal density $\exp(-H(q,p))$. The simulations have been carried out using the gle-BAOAB scheme \cite{leimkuhler2022efficient}  with step size $h=10^{-5}$.}
    \label{fig:heatbath4}
\end{figure}

\subsection{Computing Green's function by Monte Carlo}

As a second example, we consider the computation of the Green's function of a discretised elliptic differential operator on an bounded open domain, specifically we consider the symmetric operator
\[
\cL = \Delta - k^2
\]
on the unit interval $\Omega=(0,1)$, equipped with homogeneous Dirichlet boundary conditions. 
In the following, we do not distinguish between the infinite-dimensional operator $\cL$ as an operator on the domain $\mathcal{D}=H^1_0(0,1)\cap H^2(0,1)$ and its finite-difference approximation as a $d\times d$ matrix. We denote both, the linear operator and its finite-dimensional approximation by the same symbol $\cL$.  

Our aim is to obtain a Monte Carlo approximation of the Green's function, $G$, i.e.~the solution to the linear elliptic boundary value problem 
\begin{equation}\label{greens}
    \cL G(\cdot,a) = \delta_a\,,
\end{equation}
where $\delta_a$ denotes the Dirac delta function centered at $a\in(0,1)$. 
To this end we adopt ideas put forward in \cite{white2009green} and consider the OU process
\begin{equation}\label{elliptic}
    dX_t = (\cL X_t - g)\,dt + \sqrt{2}dW_t\,.
\end{equation}
Since the noise term is non-degenerate and $\cL$ is symmetric negative definite, the process converges to a unique Gaussian invariant measure with mean $\mu=\cL^{-1}g$ and covariance $-\cL^{-1}$. Setting $g=\delta_a$, we see that the mean formally agrees with the solution of the linear boundary value problem (\ref{greens}). 

For practical computations, we now discretise the unit interval $[0,1]$ into $d-1$ equally spaced intervals of length $\Delta u=(d-1)^{-1}$ and discretise the operator $\cL$ by finite differences, which yields 
\[
\cL = \frac{1}{(\Delta u)^2}\begin{pmatrix}
		-2 & 1 & 0 & \ldots & 0 \\ 1 &  -2  & 1 & \ldots & 0\\ 0 & \ddots & \ddots & \ddots & 0\\ \vdots & \vdots & \ddots & -2  & 1 \\ 0 & \ldots & 0 & 1 & -2  
	\end{pmatrix} - k^2 \begin{pmatrix}
		1 & 0 & 0 & \ldots & 0 \\ 0 &  1  & 0 & \ldots & 0\\ 0 & \ddots & \ddots & \ddots & 0\\ \vdots & \vdots & \ddots & 1 & 0 \\ 0 & \ldots & 0 & 0 & 1  
	\end{pmatrix}\in\R^{d\times d}\,.
\]
By construction, the matrix is Hurwitz, yet the solution does not satisfy the desired boundary conditions. Therefore we impose the constraint 
\[
X_1 = X_d = 0
\]
to implement homogeneous Dirichlet boundary conditions on the solution. Note that $X_1$ and $X_d$ constitute the constrained variable $X^1$, while the interior nodes $X_2,\ldots,X_{d-1}$ form the vector $X^2$. Yet, we refrain from permuting the states in order for the matrix $\cL$ to have the standard form of a discrete differential operator. The Green's function is then approximated by computing the stationary (ergodic) mean under hard and soft constraints where the softly constrained dynamics is governed by the SDE
\begin{equation}\label{elliptic_eps}
    dX_t = (\cL X_t - g)\,dt - \frac{1}{\eps}K X_t + \sqrt{2}dW_t\,,
\end{equation}
with 
\[
K = \begin{pmatrix}
		1 & 0 & 0 & \ldots & 0 \\ 0 &  0  & 0 & \ldots & 0\\ 0 & \ddots & \ddots & \ddots & 0\\ \vdots & \vdots & \ddots & 0  & 0 \\ 0 & \ldots & 0 & 0 & 1  
	\end{pmatrix}\in\R^{d\times d}\,.
\]
Since the resulting projection is orthogonal and the overall dynamics is reversible, the invariant measure of the constrained dynamics agrees with the soft constraint limit.  

Note that, since the first and the last component, $X_1$ and $X_d$, represent the boundary values of the sampled Green's function, i.e.~the constrained variables, the upper left and lower right entries of the matrix $K$ together form the invertible $K_{11}$ block of Theorem \ref{thm:genOU-SC}.    

\begin{figure}
    \centering
    \includegraphics[width=0.495\linewidth]{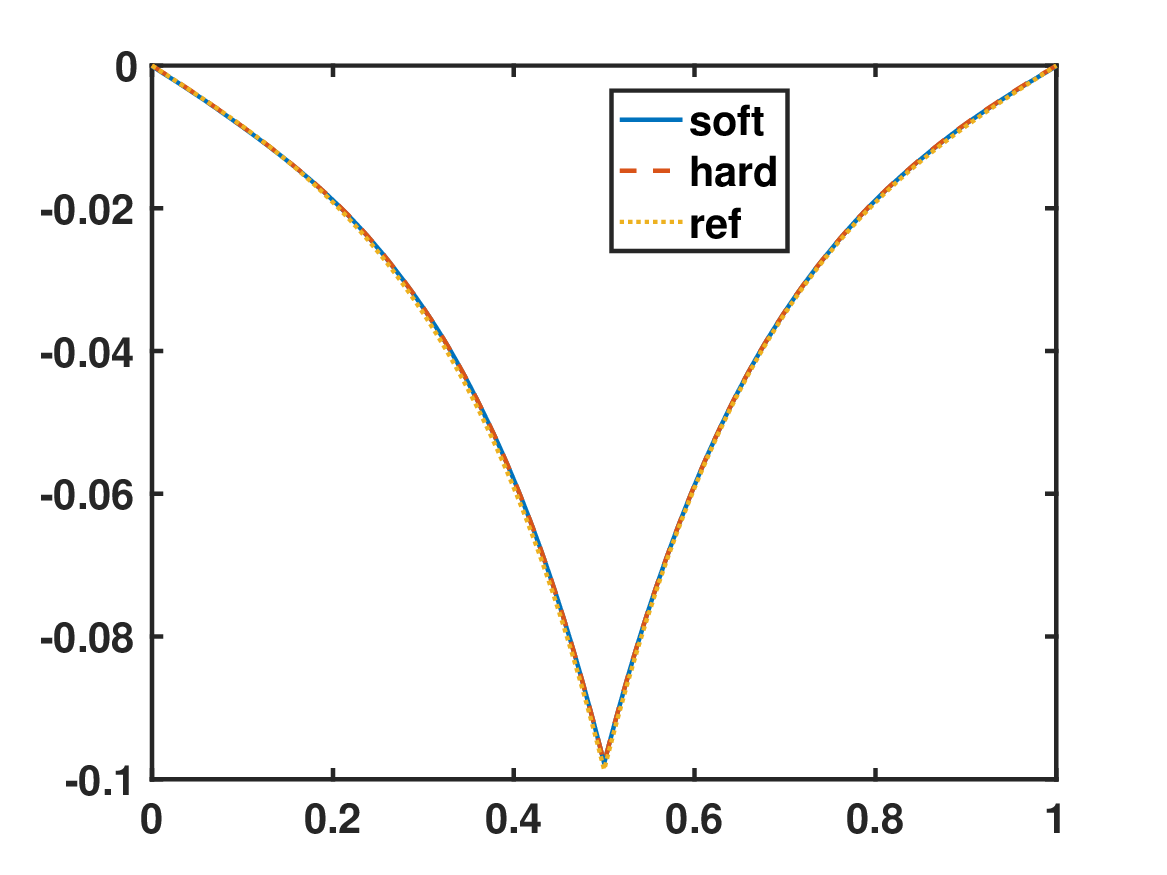}
    \includegraphics[width=0.495\linewidth]{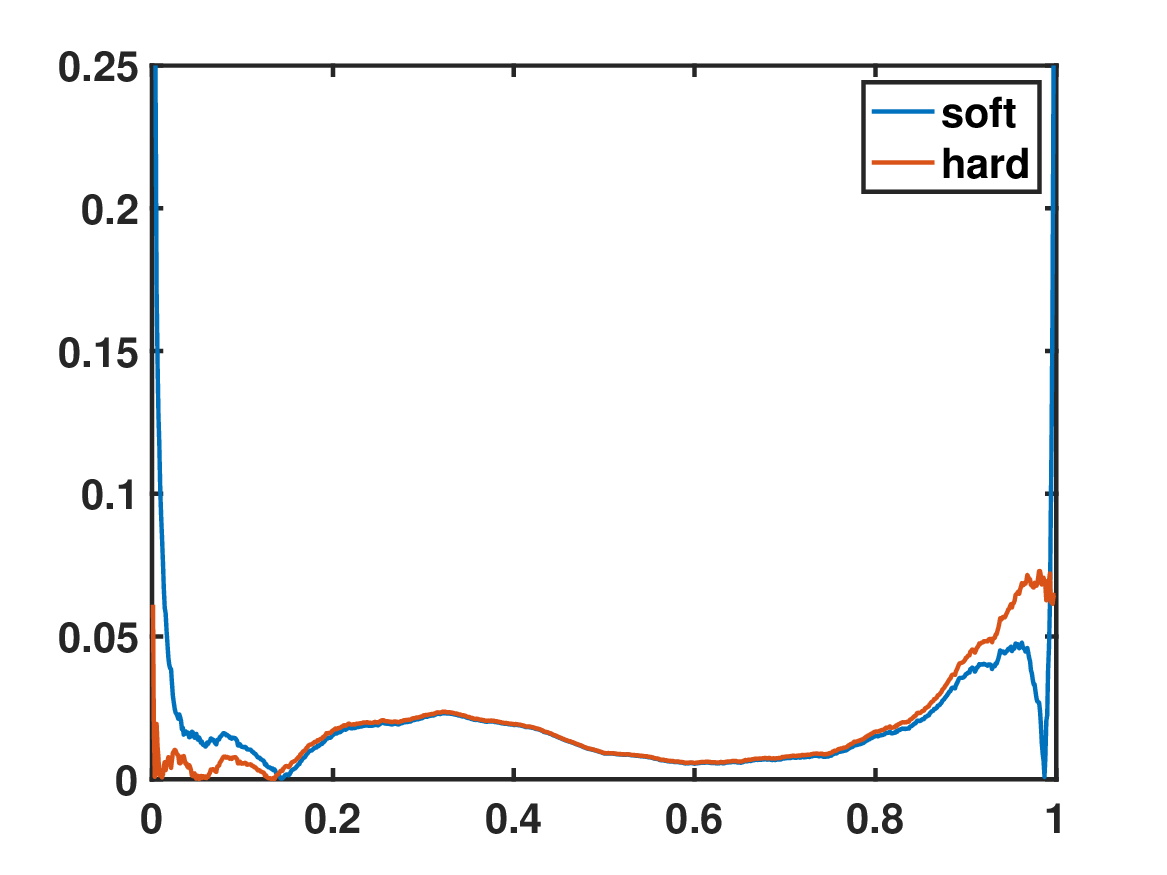}
    \caption{Numerically computed Green's functions of the elliptic operator $\cL=\Delta-k^2$ with homogeneous Dirichlet boundary conditions (left panel) and the resulting relative errors (right panel).}
    \label{fig:bvp}
\end{figure}

Figure \ref{fig:bvp} shows the running average of the hard and soft constraint solutions (i.e.~the ergodic mean)  for $\Delta u=10^{-3}$ together with the exact Green's function
\[
G(u,a) =\displaystyle
\begin{cases} \displaystyle
\exp(-k) \frac{2k \, \exp(ka) - \exp(k(2-a))}{\exp(k) - \exp(-k)} (\exp(ku) - \exp(-ku)) & \text{if } u \leq a, \\[2ex]\displaystyle
\exp(-k) \frac{2k \, \exp(ka) - \exp(-ka)}{\exp(k) - \exp(-k)} (\exp(ku) - \exp(k(2-u))) & \text{if } u > a\,.
\end{cases}
\]
The Dirac delta function in (\ref{elliptic}) and (\ref{elliptic_eps}) is approximated by 
\[
g = (\Delta u)^{-1}e_{\lceil ad\rceil}  \,,
\]
where $e_k\in\R^d$ denotes the $k$-th unit vector, and $\lceil s\rceil$ is the smallest integer greater than $s\in\R$.  
The numerical solutions have been computed by averaging over trajectories of length $T=10$ using an Euler-Maruyama discretisation with step size $h=10^{-7}$. The trajectory is subsampled using a macro time step $\Delta t=10^{-4}$ (The small step size $h$ is required since the largest eigenvalue of the matrix $-\cL$ is approximately $4\cdot 10^6$.) The confinement parameter $\eps$ has been set to $10^{-3}$. 
The right panel of Figure \ref{fig:bvp} shows the relative error
\[
\frac{|G_{\rm num}(u,a) - |G(u,a)|}{|G(u,a)|}\,,
\]
which demonstrates good agreement of both hard and soft constraint averages with the exact solution, except at the boundaries $u\in\{0,1\}$ where the relative error is inevitably large, since $G(0,a)=0=G(1,a)$.

\section{Conclusions}\label{sec:conclusions}

We have studied the realisation of constraints on linear stochastic differential equations (SDE) by strong confining forces. Specifically, we have considered affine constraints and proved that the dynamics converges pathwise to the solution of a linear SDE that lives solely on the constraint subspace in the limit of infinitely strong confinement. The limit SDE can be recast as a linear projection of the original SDE onto the constraint subspace where the projection can be orthogonal or oblique, depending on the choice of the confining force. Under certain assumptions, the original SDE has a unique Gaussian invariant measure, and we have given necessary and sufficient conditions for the confining force and the resulting projection matrix that guarantee that the constrained dynamics preserves the invariant measure (on the constraint subspace). We have illustrated the theoretical findings with two numerical examples: a linear Hamiltonian system that is coupled to a stochastic heat bath, with a constraint imposed on the heat bath variables, and a linear SDE with a drift matrix that is a discretised elliptic differential operator on a one-dimensional domain, with the constraint representing boundary conditions.

Future work ought to address the partially or fully nonlinear case, i.e.~nonlinear SDEs with linear or nonlinear constraints. One case that is of particular relevance, both from a theoretical point of view and for practical purposes, is the underdamped Langevin system with different confinement mechanisms (e.g.\ stiff springs, high friction, or large masses) that act on configuration and/or momentum variables. Langevin systems will be dealt with in a forthcoming paper~\cite{HartmannNeureitherSharma25}.

\section*{Acknowledgement} The authors thank Scott Hottovy and Mark Peletier for helpful discussions on the compactness argument related to the work \cite{Katzenberger91}. This research has been partially funded by the German Federal Government, the Federal Ministry of Education and Research and the State of Brandenburg within the framework of the joint project \emph{EIZ: Energy Innovation Center} (project numbers 85056897 and 03SF0693A) and by the Deutsche Forschungsgemeinschaft (DFG) through the grant \emph{CRC 1114: Scaling Cascades in Complex Systems} (project no. 235221301).

\begin{appendices}

\section{Variation of constants}\label{app:VarConst}

Throughout the  proofs in this article we will make use of the following simple result which summarises an explicit solution for a class of SDEs with linear and nonlinear drift terms.

\begin{prop}[Variations of constants] \label{prop:varofconst}
For $t>0$, let $(x_t,y_t)\in\R^\ell\times\R^m$ be a strong unique solution to a coupled SDE system, where $x_t$ evolves according to  
\begin{align} \label{eq:SDEmatrixnot}
dx_t = A x_t dt+ f(x_t,y_t,t) dt + C dW_t.
\end{align} 
Here $A \in \R^{\ell \times \ell}$ and $C \in \R^{\ell \times \ell}$ are constant matrices, $W_t$ is a standard Brownian motion in $\R^\ell$ and  $x_0\in\R^\ell$ is the initial condition. Furthermore assume that $f\colon \R^{\ell+m}  \times \R_{\geq 0} \to \R^\ell$ satisfies
\begin{enumerate}
    \item uniform Lipschitz continuity in space, i.e.\ there exists $L_f>0$ such that for any $z,z'\in\R^{\ell+m}$ we have  $|f(z,t)-f(z',t)|\leq L_f|z-z'|$;
    \item sublinear growth, i.e.\ there exists a constant $c>0$ such that for any $z\in\R^{\ell+m}$ and $t>0$ we have $|f(z,t)|\leq c(1+|z|)$. 
\end{enumerate}
Then $x_t$ can be explicitly written as
\begin{equation}\label{eq:VarofCon-sol}
    x_t = e^{At}x_0 + \int_0^t e^{A(t-s)} f(x_s,y_s,s) ds + \int_0^t e^{A(t-s)} C dW_s\,.
\end{equation}
\end{prop}

The existence and uniqueness of the strong solution is standard (see for instance~\cite[Theorem 26.8]{Klenke13}). The integral form of the solution follows by using variations of constants along with integration by parts for It\^o integrals. 
\begin{proof}
Using the stochastic integration by parts formula (see for instance~\cite[Theorem 18.16]{Kallenberg21}) and $W_0=0$ almost surely we find
\begin{align}\label{eq:stoc-IntByParts}
    \int_0^t e^{A(t-s)} C dW_s = C W_t + \int_0^t A e^{A(t-s)} C W_s ds 
\end{align}
almost surely. Hence $x_t - CW_t \eqqcolon g_t$, where $x_t$ is given by~\eqref{eq:VarofCon-sol}, is differentiable and satisfies 
\begin{equation*}
    \dot g_t = A \biggl( e^{At}x_0 + \int_0^t e^{A(t-s)} f(x_s,y_s,s) ds +  \int_0^t A e^{A(t-s)} C W_s ds \biggr) + f(x_t,y_t,t) + ACW_t = A x_t + f(x_t,y_t,t).
\end{equation*}
Therefore using the definition of $g_t$ we have
\begin{align*}
    x_t = g_0 + \int_0^t  \frac{dg_s}{ds} ds + CW_t 
        = x_0 + \int_0^t \bigl[Ax_s + f(x_s,y_s,s)\bigr]ds + \int_0^t C\, dW_s,
\end{align*}
i.e.\ $x_t$ solves~\eqref{eq:SDEmatrixnot}. 
\end{proof}

\section{Multivariate Ornstein-Uhlenbeck processes} \label{app:OU}
In this section we present fundamental results on pathwise solutions of OU processes and the existence of a unique invariant measure. Consider a general OU process
\begin{align} \label{eq:OUgeneral}
d X_t = M(X_t + v) dt + \sqrt{2} C dB_t
\end{align}
where $X_t \in \R^d, \, M \in \R^{d \times d}, \, \rank(M)=d, \, C \in \R^{d \times m}, \, \rank(C) \leq d, \, v\in\R^d$ and $B_t $ is $d$-dimensional Brownian motion. 
Using variation of constants (see Proposition~\ref{prop:varofconst}) it follows that the solution to \eqref{eq:OUgeneral} is 
    \begin{align*}
        X_t = e^{Mt} X_0 + (e^{Mt} - I)v +\sqrt{2} \int_0^t e^{M(t-s)} C dB_s\,.
    \end{align*}
Using $m_t=\E[X_t] \in\R^d$ and $\Sigma_t=\E[(X_t-m_t)((X_t-m_t))^T]\in\R^{d\times d}$ for the mean and variance respectively and assuming that $X_0 \sim \mathcal{N}(m_0, \Sigma_0)$ we find that $X_t \sim \mathcal{N}(m_t, \Sigma_t)$ with 
    \begin{align}\label{linSDE-NormalSolution}
        m_t = e^{Mt} m_0  + e^{Mt}v - v, \ \Sigma_t = e^{Mt} \Sigma_0 e^{M^T t} + 2 \int_0^t e^{M(t-s)} CC^T e^{M^T (t-s)} ds.  
    \end{align}
Here $\mathcal N(\cdot,\cdot)$ is the normal distribution. 
The next proposition states the assumptions under which~\eqref{eq:OUgeneral} admits a unique invariant measure.  
\begin{prop} \label{prop:invmeasOU}
The OU process defined by \eqref{eq:OUgeneral} admits a unique invariant measure if and only if the following two conditions are satisfied
\begin{itemize}[label=\roman*]
    \item[(i)] $M$ is Hurwitz,
    i.e.\ its spectrum lies in the open left-plane;
    \item[(ii)] $(M,C)$ is controllable, i.e.,\
    $\mathrm{rank}[C, MC, M^2C,\ldots, M^{d-1}C]=d$.
\end{itemize}
The invariant measure $\mu$ is given by
\begin{align*}
    \mu = \mathcal{N}(-v,\Sigma), 
\end{align*}
where $\Sigma=\Sigma^T >0$ is the unique solution to the Lyapunov equation
\begin{align}\label{eq:lyapunovgen}
    M \Sigma + \Sigma M^T = -2 CC^T. 
\end{align}
\end{prop}
\begin{proof}
Note that $e^{Mt} \to 0$ as $t \to \infty$ since $M$ is Hurwitz. Consequently $\lim\limits_{t \to \infty} m_t = -v$ and $\Sigma \coloneqq \lim\limits_{t \to \infty} \Sigma_t = 2 \int_0^\infty e^{Ms} CC^T e^{M^Ts} ds$ if and only of $M$ is Hurwitz. 
Moreover $\Sigma >0$ if and only if $(M,C)$ is controllable \cite[Theorem 1.6]{zabczyk2020mathematical}. As $\Sigma_t$~\eqref{linSDE-NormalSolution} solves the differential equation 
\begin{equation*}
     \frac{d}{dt} \Sigma_t = M \Sigma_t + \Sigma_t M^T + 2 CC^T,
\end{equation*}
the limiting covariance $\Sigma$ is the steady state of this equation, i.e.\ it satisfies~\eqref{eq:lyapunovgen}.
\end{proof}

The following result shows that the OU process~\eqref{eq:linSDE-v0} can be transformed into the form~\eqref{eq:OU_PHS} used throughout this paper. 

\begin{prop}\label{prop:genOU-PHS}
Consider the linear SDE 
\begin{equation}\label{eq:linSDE-v0}
    dX_t = M(X_t+v) dt + \sqrt{2} C dB_t 
\end{equation}
where $v\in\R^d$,  $M\in\R^{d\times d}$ is Hurwitz, $C\in \R^{d\times d}$ such that the matrix pair $(M,C)$ is controllable and $W_t$ is the standard $d$-dimensional Brownian motion. Equation~\eqref{eq:linSDE-v0} can be rewritten as
\begin{equation} \label{eq:OU_PHS}
    dX_t = (J-A)\Sigma^{-1} (X_t+v) dt + \sqrt{2}CdB_t,
\end{equation}
where $CC^T\eqqcolon A\in\R^{d\times d}$ is symmetric positive semi-definite, $J\in\R^{d\times d}$ is skew symmetric and $\Sigma \in \R^{d \times d}$ is symmetric positive definite, i.e.\ $A=A^T\geq 0, \ J=-J^T$ and $\Sigma=\Sigma^T>0$.

Conversely, for given $J=-J^T, \, \Sigma=\Sigma^T>0$, and $C \in \R^{d \times d}$ which defines $A=CC^T$, ~\eqref{eq:OU_PHS} can be reformulated as~\eqref{eq:linSDE-v0} using  
\begin{equation} \label{eq:driftdecomp}
    M = -(CC^T - J)\Sigma^{-1}\,.
\end{equation}
Furthermore, if $(M,C)$ is controllable, then $M$ is Hurwitz.       
\end{prop}
\begin{proof}
Proposition~\ref{prop:invmeasOU} implies the existence of a symmetric positive definite matrix $\Sigma$ (covariance for the invariant measure) which satisfies the Lyapunov equation~\eqref{eq:lyapunovgen}. Multiplying the Lyapunov equation from the right by $\Sigma^{-1}$ we find 
\begin{equation*}
    M= -\bigl(\Sigma M^T + 2CC^T\bigr)\Sigma^{-1}.  
\end{equation*}
Next we introduce the skew-symmetric matrix $J$ 
\begin{equation*}
    J\coloneqq \frac12 \bigl(-\Sigma M^T+M\Sigma\bigr) = -\Sigma M^T +\frac12\bigl(\Sigma M^T+M\Sigma\bigr) = -\Sigma M^T - CC^T, 
\end{equation*}
using which we can write the drift in~\eqref{eq:linSDE-v0} as 
\begin{equation*}
    MX= -\bigl(\Sigma M^T + 2CC^T\bigr)\Sigma^{-1}X = \bigl(J-CC^T \bigr)\Sigma^{-1}X.
\end{equation*}
This leads to the reformulated SDE~\eqref{eq:OU_PHS} with $A\coloneqq CC^T$.

For the converse statement, note that $M$ defined in~\eqref{eq:driftdecomp} satisfies the the Lyapunov equation~\eqref{eq:lyapunovgen}. Since $\Sigma$ is symmetric positive definite, we can define $\bar M = \Sigma^{-\frac{1}{2}}M\Sigma^{\frac12}$ with positive definite $\Sigma^{\frac12},\Sigma^{-\frac12}$. Note that $M$ and $\bar M$ have the same spectrum with the property that if $(v,\lambda)$ is an eigenpair for $M^T$ then $(\bar v,\lambda)$, with $\bar v=\Sigma^{\frac12}v$, is an eigenpair for $\bar M^T$. Therefore using~\eqref{eq:driftdecomp} we find (with $v^*$ as the complex conjugate of $v$) 
\begin{align*}
     \lambda |\bar v|^2 = \lambda \bar v^* \bar v = \bar v^* \bar M^T \bar v = - \bar v^* \Sigma^{-\frac12} CC^T \Sigma^{-\frac{1}{2}} \bar v + \bar v^*\Sigma^{-\frac12} J^T \Sigma^{-\frac{1}{2}} \bar v = - v^*  CC^T v - v^* J v.
\end{align*}
Note that the controllability of $(M,C)$ is equivalent to the statement that no eigenvector of $M^T$ is in the kernel of $C^T$ (see~\cite[Theorem 1.6]{zabczyk2020mathematical} and~\cite[Lemma 2.3]{arnold2014sharp}), i.e.\ $v^*  CC^T v>0$. Using $v^* J v=0$ since $J= - J^T$ and taking the real part of the equation above we find that $\mathrm{Real}(\lambda)<0$, i.e.\ $M$ is Hurwitz. 
\end{proof}

\section{Limit of certain matrix exponentials}\label{app:MatrixExp}

\begin{proof}[Proof of Lemma \ref{lem:matrixexp}]
\ref{lem:matrixexpstatement} Let $\xi\colon \R^d \to \R^k$, $\xi(x) =x^1 - b \in \R^k$, so that $\nabla \xi = \begin{pmatrix}
    I_{k \times k} \\ 0_{(d-k) \times k}\end{pmatrix}. $ Also introduce $V \coloneqq \begin{pmatrix}
        0_{k \times k}\\ I_{(d-k) \times k} \end{pmatrix}\,$ and note that  $\nabla \xi^T V = 0 \in \R^{k\times k}$. Computing 
    \begin{align*}
            e^{-\tilde K t}=\exp\left(- \begin{pmatrix} K_{11} & 0 \\ K_{21} & 0 \end{pmatrix} t \right) = \exp\left(- K \nabla \xi \nabla \xi^T t \right)
             = I_{d \times d}  + \sum\limits_{j \geq 1} \left(- t\right)^j (K \nabla \xi \nabla \xi^T)^j, 
    \end{align*}
   it follows that $e^{-\tilde K t}V=V$. Similarly,  noting that $\nabla \xi^T K \nabla \xi = K_{11} $, we have  $e^{-\tilde K t} K \nabla \xi = K \nabla \xi e^{-K_{11}t}.$
   Writing \[[A,B] = \begin{pmatrix}
       a_{11}& \cdots &a_{1n} & b_{11} & \cdots & b_{1m} \\
       \vdots &      &       &        &      &  \vdots \\
       a_{r1} & \cdots & a_{rn} & b_{r1} & \cdots & b_{rm} 
   \end{pmatrix} \in \R^{r\times (n+m)}\] for two matrices $A  \in \R^{r \times n} , \ B \in \R^{r \times m}$ with entries  $a_{ij}, b_{il} $ for $1\leq i \leq r, 1\leq j \leq n, 1\leq l \leq m, $  respectively, the last two equations can be combined to give
         $e^{-\tilde K t}\left[ V ,  K \nabla \xi \right] = \left[ V ,  K \nabla \xi e^{-K_{11}t}  \right] $ or equivalently
 \begin{align}
&e^{-\tilde K t}= \left[ V ,  K \nabla \xi e^{-K_{11}t}  \right] \left[ V ,  K \nabla \xi \right]^{-1}.
    \end{align}
    Observe that 
    $
        \left[ V ,  K \nabla \xi \right] = \begin{pmatrix}
            0_{k \times k} & K_{11} \\
            I_{(d-k) \times  k)} & K_{21} 
        \end{pmatrix}$
        and its inverse is given by
       $  \left[ V ,  K \nabla \xi \right]^{-1} = 
        \begin{pmatrix}
            -K_{21}K_{11}^{-1} & I_{k \times (d-k)} \\ K_{11}^{-1} & 0
        \end{pmatrix}
    $
    as can be checked. Therefore
    \begin{align*}
       e^{-\tilde K t} &= \left[ V ,  K \nabla \xi e^{-K_{11}t} \right] \left[ V ,  K \nabla \xi \right]^{-1} 
        = \begin{pmatrix}
            0 & K_{11}e^{- K_{11} t} \\  I & K_{21} e^{- K_{11} t}
        \end{pmatrix} \begin{pmatrix}
            -K_{21}K_{11}^{-1} & I_{k \times (d-k))} \\ K_{11}^{-1} & 0
        \end{pmatrix} \\
        &= \begin{pmatrix}
            K_{11} e^{-K_{11}t} K_{11}^{-1} & 0 \\ -K_{21} K_{11}^{-1}
 + K_{21} e^{-K_{11}t} K_{11}^{-1} & I        \end{pmatrix}
 = \begin{pmatrix}
            e^{-K_{11}t}  & 0 \\ -K_{21} K_{11}^{-1}
 + K_{21} e^{-K_{11}t} K_{11}^{-1} & I        \end{pmatrix}\,,
    \end{align*}
where the last equality follows noting that $K_{11} e^{-K_{11}t} K_{11}^{-1} = e^{-K_{11}t}$.

\noindent\ref{lem:matrixexpconvrate} If $K_{11}$ has eigenvalues with positive real parts (i.e.\ $-K_{11}$ is Hurwitz) then $e^{-K_{11}t} \to 0_{k \times k}$ as $ t \to \infty$ and therefore 
      \begin{align*}
        e^{-\frac1\eps \tilde K t}
        \xrightarrow{\eps\to 0} P =\begin{pmatrix}
            0 & 0 \\ -K_{21} K_{11}^{-1} & I
            \end{pmatrix}.
    \end{align*}     
    
    \noindent\ref{lem:integratedexpmatrix-projection} Using the explicit firm of the matrix exponential we have   
    \begin{align}
        \bigl\|e^{- \frac{1}{\eps} \tilde K (t-s) }  - P \bigr\|_F^2 
         &= \tr\bigl(e^{-\frac{1}{\eps}K^T_{11}(t-s)} e^{-\frac{1}{\eps}K_{11}(t-s)} + K_{11}^{-T} e^{-\frac{1}{\eps}K_{11}^T(t-s)} K_{21}^T K_{21} e^{-\frac{1}{\eps}K_{11}(t-s)} K_{11}^{-1}\bigr) \notag\\
        &=  \tr\bigl(e^{-\frac{1}{\eps}K^T_{11}(t-s)} e^{-\frac{1}{\eps}K_{11}(t-s)}   \bigr)+  \tr\bigl(K_{11}^{-T} e^{-\frac{1}{\eps}K_{11}^T(t-s)}  K_{21}^T K_{21} e^{-\frac{1}{\eps}K_{11}(t-s)} K_{11}^{-1}\bigr) \notag\\
        &= \bigl\| e^{-\frac{1}{\eps}K_{11}(t-s)}\bigr\|_F^2+  \bigl\|K_{21}e^{-\frac{1}{\eps}K_{11}(t-s)}  K_{11}^{-1}\bigr\|_F^2 \notag\\
        &\leq \bigl\| e^{-\frac{1}{\eps}K_{11}(t-s)}\bigr\|_F^2 \Bigl( 1 + \bigl\| K_{21}\bigr\|_F^2 \bigl\| K^{-1}_{11}\bigr\|_F^2 \Bigr), \label{eq:app-Exp-pre-explicit}
    \end{align}
where we used the sub-multiplicativity of the Frobenius norm to arrive at the inequality. 
We can write 
\begin{align*}
    K_{11} = V \left( \Lambda + N \right) V^{-1}
\end{align*}
in its Jordan normal form, where $\Lambda$ is a diagonal matrix with the eigenvalues of $K_{11}$, $N$ is an upper triangular nilpotent matrix of order $m$, i.e.\ $N^m = 0$ for some $m \in \mathbb{N}$, and $V$ consists of the (generalized) eigenvectors. Using $e^{-K_{11}} = V e^{-\Lambda} e^{-N} V^{-1}$ together with the sub-multiplicativity of the Frobenius norm, we have 
\begin{align*}
    \| e^{-\frac{1}{\eps}K_{11}(t-s)} \|_F^2 &=
    \| V e^{-\frac{1}{\eps}\Lambda(t-s)} e^{-\frac{1}{\eps}N(t-s)}V^{-1} \|_F^2  \leq \| V \|_F^2  \| e^{-\frac{1}{\eps}\Lambda(t-s)} \|_F^2  \| e^{-\frac{1}{\eps}N(t-s)} \|_F^2  \| V^{-1} \|_F^2\,.
\end{align*} 
Denote the real parts of the eigenvalues of $K_{11}$ by $\lambda_i>0$ with $i=1,\ldots,k$, and assume them to be ordered, i.e. $0<\lambda_1 \leq \ldots \leq \lambda_k$. Using this we find
\begin{align*}
\bigl\| e^{-\frac{1}{\eps}\Lambda(t-s)} \bigr\|_F^2 = \tr( e^{-\frac{1}{\eps}(\Lambda + \Lambda^*) (t-s)} ) = \sum\limits_{i=1}^k e^{-\frac{2}{\eps}\lambda_i (t-s)} \leq k e^{-\frac{2}{\eps} \lambda_{1}(t-s)},
\end{align*} 
where $\Lambda^*$ is the complex conjugate of $\Lambda$. 
Next, observe that since $N$ is a  nilpotent matrix of order $m$, i.e.\ $N^m=0$, we have $e^{-\frac{1}{\eps}N} = \sum\limits_{j=0}^{m-1} \left(\frac{-N}{\eps}\right)^j\frac{1}{j!}$, which leads to 
\begin{align*}
    \| e^{-\frac{1}{\eps}N(t-s)}\|_F^2 \leq m\sum_{j=0}^{m-1} \left(\frac{(t-s)}{\eps}\right)^j\frac{1}{j!} \|N^j\|_F^2,
\end{align*}
where the constant $m$ arises due to Young's inequality. Using the condition number $\kappa(V)=\|V\|_F^2\|V^{-1}\|_F^2$ we arrive at the overall estimate
\begin{align*}
     \| e^{-\frac{1}{\eps}K_{11}(t-s)} \|_F^2 \leq \kappa(V) km e^{-\frac{2}{\eps} \lambda_1(t-s)} \sum\limits_{j=0}^{m-1}(t-s)^j \frac{\|N^j \|_F^2}{j! \eps^j}.
\end{align*}
Substituting into~\eqref{eq:app-Exp-pre-explicit} we arrive at
\begin{align}\label{eq:Frobeniusest_expK-P-Explicit}
      \| e^{-\frac{1}{\eps}K_{11}(t-s)} -P \|_F^2 \leq  \Bigl( 1 + \bigl\| K_{21}\bigr\|_F^2 \bigl\| K^{-1}_{11}\bigr\|_F^2 \Bigr) \kappa(V) k m e^{-\frac{2}{\eps} \lambda_1(t-s)} \sum\limits_{j=0}^{m-1}(t-s)^j \frac{\|N^j \|_F^2}{j! \eps^j},
\end{align}
which completes the proof of~\eqref{eq:Frobeniusest_expK-P}.

Next we consider the special case that $K_{11}$ is symmetric (consequently non-defective), and therefore all its eigenvalues are real. Then $K_{11} = O \Lambda O^T$ for some orthonormal matrix $O$ (i.e.\ $OO^T=I$)  and we arrive at 
\begin{align*}
    \|e^{-\frac{1}{\eps}K_{11}(t-s)}\|_F^2 &= \tr(O e^{-\frac{1}{\eps}\Lambda(t-s)}O^T O e^{-\frac{1}{\eps}\Lambda(t-s)}O^T) = \| e^{-\frac{1}{\eps}\Lambda(t-s)} \|_F^2 \leq k e^{-\frac{2\lambda_1}{\eps}(t-s)} ,
\end{align*} 
where $\lambda_1>0$ is the smallest eigenvalue of $K_{11}$ as before. In this case $N\equiv 0$ and using $\kappa(V)=1$ we arrive at part~\ref{item:sym-nonDef}. Note that the same arguments as above also apply for any non-defective matrix (now $\kappa(V)\neq 1$ but $N\equiv 0$) which leads to part~\ref{item:nonDef} (note that we have used the time-integral bound derived below as well).

Finally, we prove the bound~\eqref{eq:matrixexp-int-estimate} on the time integral. To achieve this we use the bound ($j\in \{0,\ldots, m-1\}$)
\begin{equation}\label{eq:calcintexptimespolynomial}
\begin{aligned} 
    \int_0^t e^{-\frac{2\lambda_1}{\eps}(t-s)} \frac{(t-s)^j}{\eps^j} ds &= \int_0^t e^{-\frac{2\lambda_1}{\eps}s} \frac{s^j}{\eps^j} ds = - \frac{t^j}{2\lambda_1 \eps^{j-1}} e^{-\frac{2\lambda_1}{\eps}t} + \frac{j}{2\lambda_1}\int_0^t e^{-\frac{2\lambda_1}{\eps}s} \frac{s^{j-1}}{\eps^{j-1}} ds  \\
    &= \ldots = - e^{-\frac{2\lambda_1}{\eps}t} \sum\limits_{\ell=1}^{j+1}\frac{ t^{j-\ell+1}}{\eps^{j-\ell}\lambda_1^\ell } \frac{j!}{(j-\ell+1)!} + \frac{j!}{(2\lambda_1)^{j+1}}\eps  \leq  \frac{j!}{(2\lambda_1)^{j+1}}\eps, 
\end{aligned}
\end{equation}
where the first equality follows via the substitution $r=t-s$ and the rest of the equalities follow by performing a series of integration by parts. 

The bound~\eqref{eq:matrixexp-int-estimate} then follows immediately by using~\eqref{eq:app-Exp-pre-explicit}-\eqref{eq:Frobeniusest_expK-P-Explicit} along with the bound above, since
\begin{align*}
    \int_0^t \| e^{-\frac{1}{\eps}K_{11}(t-s)} - P \|_F^2 ds &\leq \Bigl( 1 + \bigl\| K_{21}\bigr\|_F^2 \bigl\| K^{-1}_{11}\bigr\|_F^2 \Bigr) \kappa(V) k m\int_0^t e^{-\frac{2}{\eps}\lambda_1(t-s)}\sum_{j=0}^{m-1} \frac{(t-s)^j}{\eps^j}\frac{1}{j!}\|N^j\|_F^2 ds  \\ &\leq \eps \Bigl( 1 + \bigl\| K_{21}\bigr\|_F^2 \bigl\| K^{-1}_{11}\bigr\|_F^2 \Bigr) \kappa(V) k m\sum\limits_{j=0}^{m-1}\frac{\|N^j\|_F^2}{(2\lambda_1)^{j+1}}. 
\end{align*}
\end{proof}
\begin{proof}[Proof of Corollary \ref{cor:expK-P}]
    Let $\delta,\eps>0$ and introduce 
    \begin{align*}
        f_j(t) \coloneqq e^{-\delta t}  \beta_j  \left(\frac{t}{\eps}\right)^j\,, \ \beta_j > 0,
    \end{align*}
    where $t\in [0,\infty)$ and $j\in\mathbb{N}$. 
    Then $f_j(0) = 0$ and $\lim\limits_{t \to \infty} f_j(t) = 0$ and $f_j'(0) = \beta_j > 0$ for any $j$. Hence, there exists $t_j^* > 0$ such that $f_j(t^*_j) = \max_{t \geq 0} f_j(t) > 0$ and computing $f'_j(t) = \left(- \delta \beta_j t^j \eps^{-j} + \beta_j j t^{j-1}\eps^{-j}\right)e^{-\delta t} $ we find $t^*_j = \frac{j}{\delta}$. Hence $f_j(t^*_j) = \beta_j \left(\frac{j}{\delta \eps e}\right)^j$. With this
    \begin{align*}
        e^{-2\frac{\lambda}{\eps} t }  \sum\limits_{j=0}^{m-1} \beta_j \left(\frac{t}{\eps}\right)^j = e^{-2\frac{\lambda}{\eps} t + \delta t} \sum\limits_{j=0}^{m-1} f_j(t) \leq e^{-2\frac{\lambda}{\eps} t + \delta t} \sum\limits_{j=0}^{m-1} \beta_j   \left(\frac{j}{\delta\eps e}\right)^j\,. 
    \end{align*}
    Now let $\delta = 2\frac{\tilde \delta}{\eps}$ for   $\tilde \delta\in  (0,\lambda)$. Inserting $\beta_j = \|N^j\|_F^2 \leq \|N\|_F^{2j}$ which holds by sub-multiplicativity of the norm, we find that for any $\tilde \delta \in (0,\lambda)$
        \begin{align*}
        e^{-2\frac{\lambda}{\eps} t }  \sum\limits_{j=0}^{m-1} \|N^j\|_F^{2} \left(\frac{t}{\eps}\right)^j \leq e^{-2\frac{\lambda -\tilde \delta}{\eps} t } \sum\limits_{j=0}^{m-1} \|N^j\|_F^{2} \left(\frac{j}{2\tilde \delta e}\right)^j\, \leq e^{-2\frac{\lambda -\tilde \delta}{\eps} t } (m-1) \left(\frac{{m-1}}{2\tilde \delta e}\right)^{m-1} \|N\|_F^{2(m-1)}.
    \end{align*}
\end{proof}

\end{appendices}
{\small
\bibliographystyle{alpha}
\bibliography{sample}
}
\end{document}